\documentclass[11pt,reqno]{amsart}
\usepackage{amsthm}
\usepackage{amsmath,amssymb,txfonts,tipa,amsfonts,ascmac,mathrsfs,multicol,amscd,ascmac,fancybox}
\usepackage[dvipdfmx]{graphicx}
\usepackage[dvipdfmx]{color}
\usepackage[pagewise]{lineno}
\usepackage{amsaddr}
\usepackage[margin=1.2in]{geometry}
\usepackage{bm}

\usepackage{url}
\usepackage[initials]{amsrefs}
\usepackage{amsmath,amsthm,amscd}
\usepackage{enumerate}
\usepackage{amssymb}
\usepackage{mathrsfs}
\usepackage{comment}

\usepackage{mathtools}

\usepackage[latin1]{inputenc}
\usepackage{tikz}
\usetikzlibrary{shapes,arrows}
\usetikzlibrary{positioning}
\usetikzlibrary{intersections,calc}

\usepackage[latin1]{inputenc}
\usepackage{tikz}
\usetikzlibrary{shapes,arrows}
\usetikzlibrary{positioning}
\usetikzlibrary{intersections,calc}

\makeatletter
\@namedef{subjclassname@2010}{%
  \textup{2010} Mathematics Subject Classification}
\makeatother

\numberwithin{equation}{section}

\newtheorem{theorem}{Theorem}[section]
\newtheorem{lemma}{Lemma}[section]
\newtheorem{proposition}{Proposition}[section]

\theoremstyle{definition}
\newtheorem{definition}{Definition}[section]
\newtheorem{remark}{Remark}[section]

\newcommand{\al}{\alpha}
\newcommand{\esssup}{\operatorname*{ess\,sup}}
\newcommand{\essinf}{\operatorname*{ess\,inf}}
\newcommand{\supp}{\operatorname*{supp}}
\newcommand{\mb}[1]{\mathbb{#1}}
\newcommand{\bdsy}[1]{\boldsymbol{#1}}
\newcommand{\vep}{\varepsilon}
\newcommand{\vphi}{\varphi}
\newcommand{\R}{\mathbb{R}}
\newcommand{\lam}{\lambda}

\newcommand{\lV}{\lVert}

\newcommand{\rV}{\rVert}
\newcommand{\N}{\mathbb{N}}

\begin{document}

\baselineskip=17pt

\title[Invariant measures for random piecewise convex maps]{Invariant measures for random piecewise convex maps}

\author[T. Inoue]{Tomoki INOUE}
\address[T. Inoue]{Graduate School of Science and Engineering\\ Ehime University\\ 3, Bunkyo-cho, Matsuyama, Ehime\\ 790-8577, JAPAN}
\email[T. Inoue]{inoue.tomoki.mz@ehime-u.ac.jp}

\author[H. Toyokawa]{Hisayoshi TOYOKAWA}
\address[H. Toyokawa]{Faculty of Engineering\\ Kitami Institute of Technology\\ 165 Koen-cho, Kitami, Hokkaido\\ 090-8507, JAPAN}
\email[H. Toyokawa]{h\_toyokawa@mail.kitami-it.ac.jp}

\begin{abstract}
We show the existence of Lebesgue-equivalent conservative and ergodic  $\sigma$-finite invariant measures for a wide class of one-dimensional random maps consisting of piecewise convex maps.
We also estimate the size of invariant measures around a small neighborhood of a fixed point where the invariant density functions may diverge.
Application covers random intermittent maps with critical points or flat points.
We also illustrate that the size of invariant measures tends to infinite for random maps whose right branches exhibit a strongly contracting property on average, so that they have a strong recurrence to a fixed point.
\end{abstract}


\maketitle

\section{Introduction}\label{sec1}

For a given non-singular map on a probability space, the question whether an invariant measure, which is absolutely continuous with respect to the reference measure, exists or not is one of the fundamental problems in ergodic theory.
The same question for a random map (in terms of both annealed and quenched sense) makes sense and is also important passing through ergodic theory for Markov operators, skew-product transformations or Markov operator cocycles.
See \cites{Ar,F,Kre,Yu} and references therein.
On the one hand, if a system, whether deterministic or random, admits an absolutely continuous finite invariant measure, some classical ergodic theorems such as the Birkhoff ergodic theorem are applicable.
Moreover some limit theorems such as the central limit theorem may be expected \cites{ANV,D,D2}.
On the other hand, if a system possesses only an absolutely continuous $\sigma$-finite and infinite invariant measure, such a system is within the scope of infinite ergodic theory and has been paid attention for recent decades \cites{A,NNTY,TZ}.
Typical examples of such systems have indifferent or neutral, but weakly repelling, fixed points
and are known as an intermittent model, as well as a model of non-uniformly hyperbolic systems.
The existence of absolutely continuous $\sigma$-finite invariant measures and several statistical properties for random versions of intermittent maps thus have been recently enthusiastically studied \cites{BB,BBD,BBR,I3,T,T2}.

The subject of the present paper is a certain class of one-dimensional random dynamical systems called \emph{random piecewise convex maps} in annealed (or i.i.d.) sense (see the conditions \eqref{0}--\eqref{2} and \eqref{3} or \eqref{4} precisely).
The existence of invariant measures and their ergodic properties of the deterministic piecewise convex maps were firstly studied by Lasota and Yorke in \cite{LY} for the case when maps are uniformly expanding on the first branch and other branches have positive derivative, and then studied by Inoue in \cites{I0,I} for more general cases.
The aim of this paper is to generalize them, demonstrating that random piecewise convex maps admit Lebesgue-equivalent ergodic $\sigma$-finite invariant measures.
(For some interesting studies of random generalization of piecewise convex maps with ``position dependent'' probability measures, we refer to \cites{BG} (cf., \cite{Is}), while we do not deal with position dependent random maps but we handle random maps consisting of potentially uncountably many maps with infinite invariant measures.)
We also estimate the size of invariant measures, from which it is revealed whether the $\sigma$-finite invariant measures for random piecewise convex maps are finite or infinite.
The phenomenon that an invariant measure varies from finite to infinite as a parameter of a system varies is well-known for (deterministic) intermittent maps employed by Thaler in \cites{Th} or Liverani--Saussol--Vaienti (see \eqref{LSV} below) in \cites{LSV}.
Although our random piecewise convex maps also admit both finite and $\sigma$-finite infinite invariant measures depending on parameters and probabilities of choice of maps, some examples of them (e.g., Example \ref{subsec41}) have neither a common indifferent fixed point nor a common critical point, which is very different from deterministic cases.
The mechanism is, roughly speaking, derived from strong contracting property on average, which never occurs for deterministic systems and is unique to random dynamical systems.

We then briefly review the LSV map, named after Liverani--Saussol--Vaienti from \cite{LSV}, which has been analyzed as a simple model of intermittency, and we will compare them (and known random versions of them, see also \cites{BB,BBD,BBR}) with our random piecewise convex maps.
For $\al>0$, the LSV map $T_{\al}:[0,1]\to[0,1]$ is defined by
\begin{linenomath}
\begin{align}\label{LSV}
T_{\al}x=
\begin{cases}
x\left(1+2^{\al}x^{\al}\right) &x\in[0,\frac{1}{2}],\\
2x-1 &x\in(\frac{1}{2},1]
\end{cases}
\end{align}
\end{linenomath}
which has an indifferent fixed point 0 and Lebesgue-typical orbits would be trapped around a small neighborhood of 0 for a long time.
For this map $T_{\al}$, it is well-known that a Lebesgue-equivalent ergodic invariant measure exists and that the invariant density function is of order $x^{-\al}$ as $x\to 0$ \cites{LSV,Th,Y}.
Thus $T_{\al}$ possesses an equivalent finite invariant measure for $0<\al<1$ and an equivalent $\sigma$-finite and infinite invariant measure for $\al\ge1$.
The order of invariant measures radically affects their statistical properties, such as the central limit theorem or the mixing rate in the finite measure-preserving case (cf., \cites{BBR,G,S}) and the wandering rate in the Darling--Kac theorem or the arcsine law in the infinite measure-preserving case (cf., \cites{A,NNTY,TZ}).
Therefore it is certainly worth establishing invariant measures and analyzing the asymptotics of the invariant measures for given systems.
The asymptotics of the invariant density for $T_{\al}$ is also tightly related to the decay of the inverse images of the disconnected point $\frac{1}{2}$ by the left branch.
If we set $x_1(\al)=\frac{1}{2}$ and $x_n(\al)\in[0,\frac{1}{2})$ for $n\ge2$ such that $T_{\al}(x_{n+1}(\al))=x_n(\al)$ for $n\ge1$, then it follows from the results in \cites{Y,Th} that
\begin{linenomath}
\begin{align}\label{eq1}
x_n(\al)\sim C n^{-\frac{1}{\al}},
\end{align}
\end{linenomath}
for some $C>0$, where $a_n\sim b_n$ for positive sequences $\{a_n\}_{n=1}^{\infty},\{b_n\}_{n=1}^{\infty}\subset\R$ stands for $\lim_{n\to\infty}\frac{a_n}{b_n}=1$.
In this paper, for random piecewise convex maps, we establish the existence of Lebesgue-equivalent, conservative and ergodic $\sigma$-finite invariant measures and evaluate the asymptotics of the invariant measures, which is a generalization of results for random LSV maps as in \cites{BB,BBD}.
The advantage of our results is that we do not restrict ourselves to constituent maps of random dynamical systems to be at most countable nor expanding on average outside of a small neighborhood of the common indifferent fixed point.
As an application, we can modify a random LSV map to admit uniformly contracting branches and moreover to admit a critical or flat point around the inverse image of an indifferent fixed point (see Example \ref{subsec43}--\ref{subsec46}).
The key point in the estimate of the invariant measures for random piecewise convex maps is, on the contrary to the LSV maps, the decay of random inverse images of the disconnected point by the \emph{right} branches (see Definition \ref{def:y}, Theorem \ref{Cor1} and Theorem \ref{Cor2} precisely).
That is, one needs to take the contracting effect by right branches into account.
Indeed, the induced (random) map/the first return (random) map for (random) LSV maps satisfy uniformly expanding property (on average), whereas those for our random piecewise convex maps do not in general.
Hence we cannot expect so-called spectral decomposition method from the Lasota--Yorke type inequality any longer.
We refer to \cite{T2} for similar arguments and some background.

\subsection{Notation}\label{subsec11}

Throughout the paper, all sets and functions mentioned are measurable and any difference on null sets with respect to a measure under consideration is ignored.
As usual, $L^1(X,\lam)$, for a set $X$ with measurable structure and a measure $\lam$ over $X$, stands for the set of all $\lam$-integrable functions over $X$ where functions differing only on $\lam$-null sets are identified.
For a measurable set $A$, $1_A$ always denotes the indicator function on $A$.

Let $\{a_n\}_{n=1}^{\infty},\{b_n\}_{n=1}^{\infty}\subset\R$ be positive sequences.
The notation $a_n\sim b_n$ is explained below \eqref{eq1}.
For the notational convention, we further use $a_n\gtrapprox b_n$, equivalently $b_n\lessapprox a_n$, by meaning that there exists a constant $M>0$ independent of $n\in\N$ such that $a_n\ge Mb_n$ holds.
$a_n\approx b_n$ is used when $a_n\gtrapprox b_n$ and $a_n\lessapprox b_n$ hold.

\subsection{Organization}\label{subsec12}

The present paper is organized as follows.
In \S \ref{sec2}, we give necessary preliminaries and define random piecewise convex maps.
\S \ref{sec3} is devoted to our main result.
We establish in Theorem \ref{Thm1} the existence of Lebesgue-equivalent, conservative and ergodic $\sigma$-finite invariant measures for random piecewise convex maps.
Theorem \ref{Cor1} and Theorem \ref{Cor2} show the asymptotics of the invariant measures given in Theorem \ref{Thm1}.
We illustrate in \S \ref{sec4} several examples of random piecewise convex maps.
\S \ref{sec4} also provides some counterexample which possesses an infinite derivative.

\section{The model of random piecewise convex maps}\label{sec2}

In this section, we define random picewise convex maps.
Let $X=[0,1]$ be the unit interval equipped with the Lebsgue measure $\lam$ over the Borel $\sigma$-algebra $\mathcal{B}$ and
let $\mb{A}$ and $\mb{B}$ be some parameter regions with some measurable structure (usually they are subspaces of $\N$ or $\R$).
For each $\al\in\mb{A}$ and $\beta\in\mb{B}$, we define a non-singular maps $T_{\al,\beta}$, i.e., $\lam(T_{\al,\beta}^{-1}N)=0$ whenever $\lam(N)$=0, on $X$ by
\[
T_{\al,\beta}x=
\begin{cases}
\tau_{\al}(x) &x\in[0,\frac{1}{2}],\\
S_{\beta}(x) &x\in(\frac{1}{2},1]
\end{cases}
\]
where $\tau_{\al}:[0,\frac{1}{2}]\to X$ and $S_{\beta}:(\frac{1}{2},1]\to X$ for $\al\in\mb{A}$ and $\beta\in\mb{B}$ are injective and continuous maps with some conditions (see the conditions \eqref{0}--\eqref{2}, \eqref{3} and \eqref{4} precisely).
The standing assumption on $\{\tau_{\al}:\al\in\mb{A}\}$ and $\{S_{\beta}:\beta\in\mb{B}\}$ is the following:

\begin{enumerate}[(1)]
\setcounter{enumi}{-1}
\item \label{0}
The map $T:\mb{A}\times\mb{B}\times X\to X$; $(\al,\beta,x)\mapsto T_{\al,\beta}(x)$ is measurable with respect to each variable.
\end{enumerate}
Note that the above condition \eqref{0} is fulfilled if $\mb{A}$ and $\mb{B}$ are topological spaces endowed with their Borel structures and the maps $\mb{A}\ni\al\mapsto\tau_{\al}$ and $\mb{B}\ni\beta\mapsto S_{\beta}$ are continuous.

Our random dynamical systems are defined as random compositions of maps $\{T_{\al,\beta}:\al\in\mb{A},\ \beta\in\mb{B}\}$ with the condition \eqref{0} in the annealed sense.
In order to define our random dynamical systems, we set probability measures $\nu_{\mb{A}}$ and $\nu_{\mb{B}}$ on $\mb{A}$ and $\mb{B}$, respectively.
$\nu_{\mb{A}}^{\infty}$ denotes the infinite product of the probability measure of $\nu_{\mb{A}}$ over $\mb{A}^{\N}$.
Then, for the family of maps $\{T_{\al,\beta}:\al\in\mb{A},\ \beta\in\mb{B}\}$ and probability measures $\nu_{\mb{A}}$ and $\nu_{\mb{B}}$ over the parameter spaces $\mb{A}$ and $\mb{B}$, we consider the following transition function
\begin{linenomath}
\begin{align}\label{tp}
\mb{P} \left(x,A\right)=\int_{\mb{A}\times\mb{B}}1_A\left(T_{\al,\beta}x\right)d\nu_{\mb{A}}(\al)d\nu_{\mb{B}}(\beta)
\end{align}
\end{linenomath}
for each $A\in\mathcal{B}$ and $\lam$-almost every $x\in X$.
By the condition \eqref{0} and non-singularity of each $T_{\al,\beta}$ with respect to $\lam$, it is straightforward to see that this transition function is null-preserving, i.e., $\lam(N)=0$ implies $\mb{P}(x,N)=0$ for $\lam$-almost every $x\in X$.
Thus, we can define the corresponding Markov operator $P:L^1(X,\lam)\to L^1(X,\lam)$ (i.e., $Pf\ge0$ and $\lV Pf\rV_{L^1}=\lV f\rV_{L^1}$ for each $f\in L^1(X,\lam)$ non-negative) by
\[
\int_A Pf d\lam = \int_X f\cdot \mb{P}(\,\cdot\,,A)d\lam
\]
for each $f\in L^1(X,\lam)$ and $A\in\mathcal{B}$.
The adjoint operator of $P$ acting on $L^{\infty}(X,\lam)$ is denoted by $P^*$ which is characterized by
\[
\int_X Pf\cdot gd\lam = \int_X f\cdot P^*gd\lam
\]
for each $f\in L^1(X,\lam)$ and $g\in L^{\infty}(X,\lam)$.

In order to make more precise assumptions on random piecewise convex maps, we introduce some notations.
As in the previous section (note that $\tau_{\al}$ is not necessarily the same as the LSV map $T_{\al}\vert_{[0,\frac{1}{2}]}$ from \S \ref{sec1}), for $\bdsy{\al}=(\al_1,\al_2,\dots)\in\mb{A}^{\mb{N}}$, let $x_1^{\bdsy{\al}}=x_1\coloneqq\frac{1}{2}$ and
$x_{n+1}^{\bdsy{\al}}
\coloneqq \tau_{\al_{n-1}}^{-1}\circ\cdots\circ\tau_{\al_1}^{-1}(x_1^{\bdsy{\al}})$ for $n\ge1$.
For simplicity, let $x_0^{\bdsy{\al}}=x_0\coloneqq 1$ and set $X_n^{\bdsy{\al}}\coloneqq(x_{n+1}^{\bdsy{\al}},x_n^{\bdsy{\al}}]$ for $\bdsy{\al}\in\mb{A}^{\mb{N}}$ and $n\ge0$.

For considering the inverses by the right branch of $x_n^{\bdsy{\al}}$'s as well, we need the following definition:
\begin{definition}\label{def:y}
$\eta$ is defined by a map from $\mb{A}^{\mb{N}}\times\mb{B}$ to $\N\cup\{0\}$ satisfying $S_{\beta}(1)\in X_{\eta(\bdsy{\al},\beta)}^{\bdsy{\al}}$.
\end{definition}
We always assume that $\eta$ is measurable as a standing hypothesis.
Then, for $\bdsy{\al}\in\mb{A}$ and $\beta\in\mb{B}$, let $y_1^{\bdsy{\al},\beta}=y_1\coloneqq 1$ and $y_{n+1}^{\bdsy{\al},\beta}$ be the inverse of $x_{\eta(\bdsy{\al},\beta)+n}^{\bdsy{\al}}$ by the right branch of $T_{\al,\beta}$, namely,
\[
y_{n+1}^{\bdsy{\al},\beta}\coloneqq S_{\beta}^{-1}\left(x_{\eta(\bdsy{\al},\beta)+n}^{\bdsy{\al}}\right)
\quad\text{for $n\ge1$.}
\]
We set $Y_n^{\bdsy{\al},\beta}\coloneqq (y_{n+1}^{\bdsy{\al},\beta},y_n^{\bdsy{\al},\beta}]$ for $n\ge1$ and $Y\coloneqq[\frac{1}{2},1]$.

Throughout the paper we assume, together with the condition \eqref{0}, that a family of maps $\{T_{\al,\beta}:\al\in\mb{A},\ \beta\in\mb{B}\}$ satisfies the following conditions (piecewise convex property, see also Figure \ref{fig1}):
for $\nu_{\mb{A}}$-almost every $\al\in\mb{A}$ and $\nu_{\mb{B}}$-almost every $\beta\in\mb{B}$,
\begin{enumerate}[(1)]
\item\label{1}
$\tau_{\al}$ and $S_{\beta}$ are $C^1$-functions and $S_{\beta}$ can be extended to a continuous function on $Y$ (the extension is also denoted by the same symbol $S_{\beta}$) with $\tau_{\al}(0)=0$, $\tau_{\al}(\frac{1}{2})=1$ and $S_{\beta}(\frac{1}{2})=0$;
\item\label{2}
$\tau_{\al}'$ and $S_{\beta}'$ are non-decreasing on $(0,\frac{1}{2})$ and $(\frac{1}{2},1)$, respectively, with $\tau_{\al}'(0)\ge1$, $\tau_{\al}'(x)>1$ for $x\in(0,\frac{1}{2})$, $S_{\beta}'(\frac{1}{2})\ge0$ and $S_{\beta}'(x)>0$ for $x\in (\frac{1}{2},1)$, where $\tau_{\al}'(0)$ and $S_{\beta}'(\frac{1}{2})$ are taken as the right derivatives.
\end{enumerate}
By our assumptions \eqref{1} and \eqref{2}, for $\nu_{\mb{A}}^{\infty}$-almost every $\bdsy{\al}=(\al_1,\al_2,\dots)\in\mb{A}^{\N}$ and $\nu_{\mb{B}}$-almost every $\beta\in\mb{B}$, we have $T_{\al_n,\beta}X_n^{\bdsy{\al}}=X_{n-1}^{\bdsy{\al}}$ and $T_{\al,\beta}Y_{n+1}^{\bdsy{\al},\beta}=X_{\eta(\bdsy{\al},\beta)+n}^{\bdsy{\al}}$ for any $n\ge1$ and $T_{\al_0,\beta}Y_1^{\bdsy{\al},\beta}=(x_{\eta(\bdsy{\al},\beta)+1}^{\bdsy{\al}},S_{\beta}(1)]\subset X_{\eta(\bdsy{\al},\beta)}^{\bdsy{\al}}$, where $\al_0\in\mb{A}$ is arbitrary.

\begin{remark}
(I)
The phase space $X$ is of course not necessarily $[0,1]$ but just needs to be a bounded interval in $\R$.
Similarly, the choice of the discontinuous point $\frac{1}{2}$ is just for simplicity, that is, we can take an arbitrary $c\in(0,1)$ instead of $\frac{1}{2}$ so that $\tau_{\al}$ and $S_{\beta}$ are defined on $[0,c]$ and $(c,1]$ respectively.
Other similar generalizations, such as increasing the number of partitions to more than two or the case when $\tau_{\al}$'s are not surjective too, may be considered without big difficulty.
For instance, if we decompose $X$ into $\{X_i\}_{i=0}^n$ with $X_i=[a_i,a_{i+1})$ with $0=a_0<a_1<a_2<\cdots<a_n<a_{n+1}=1$ for some $n\ge2$, where maps on $X_0$ satisfy the conditions on $\tau_{\al}$ and maps on $X_i$ for $i=1,\dots,n$ satisfy the conditions on $S_{\beta}$ from \eqref{1} and \eqref{2}, then the strongest contracting property in $\{X_i\}_{i=1}^n$ would dominate the statistical laws of the random system.

(II)
In the condition \eqref{1}, the assumption that $\tau_{\al}$ and $S_{\beta}$ are $C^1$ can be relaxed to the following condition:
there are families of countable open subintervals $\{I_{n}^{L}\}_n$ and $\{I_{n}^{R}\}_n$, with the closure of their union being $X$, such that, for $\nu_{\mb{A}}$-almost every $\al\in\mb{A}$ and $\nu_{\mb{B}}$-almost every $\beta\in\mb{B}$, $\tau_{\al}$ and $S_{\beta}$ are $C^1$ on $I_n^L$ and $I_n^R$, respectively for each $n$.
Hence some (but not all) examples from \cite{T2} are also in sight of the present paper.

(III)
In the above conditions \eqref{1} and \eqref{2}, we do not exclude $\tau_{\al}'(0)=1$ nor $S_{\beta}'(\frac{1}{2})=0$ for $\al\in\mb{A}$ and $\beta\in\mb{B}$.
Furthermore, we will consider a random map with the common indifferent fixed point and the common flat point i.e., $\tau_{\al}'(0)=1$ and $S_{\beta}^{(n)}(\frac{1}{2})=0$ for any $\al\in\mb{A}$, $\beta\in\mb{B}$ and $n\ge1$ (see Example \ref{subsec44} and \ref{subsec45}).
\end{remark}

\begin{figure}[h]
\centering
\includegraphics[width=8cm]{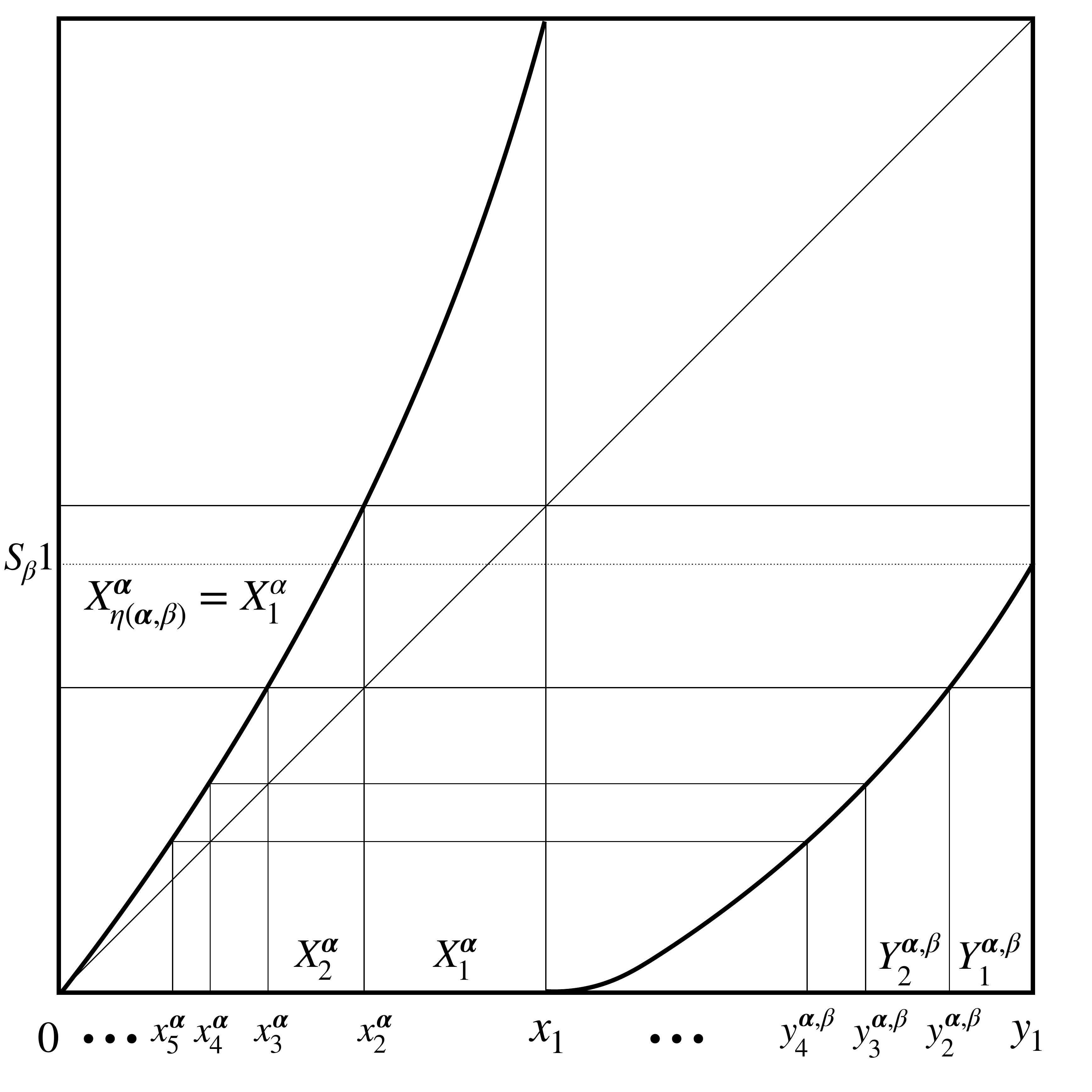}
\caption{A graph of a possible $T_{\al,\beta}$: in this case $\eta(\bdsy{\al},\beta)=1$ where $\bdsy{\al}=(\al,\al,\dots)$.}
\label{fig1}
\end{figure}

Recall that for a Markov operator $P$, a measure $\mu$ over $(X,\mathcal{B})$ is called an \emph{absolutely continuous} (resp.\ \emph{equivalent}) \emph{$\sigma$-finite invariant measure} if $\mu$ is a $\sigma$-finite measure which is absolutely continuous (resp.\ equivalent) with respect to $\lam$ and its Radon--Nikod\'{y}m derivative $\frac{d\mu}{d\lam}$ is a (non-trivial) fixed point of $P$.
Note that by positivity of Markov operators the domain of any Markov operator can be naturally extended to the set of non-negative and locally integrable functions and hence the definition of absolutely continuous $\sigma$-finite \emph{infinite} invariant measures makes sense.

We then consider the following technical conditions on our random dynamical systems, which are important in establishing the existence of equivalent $\sigma$-finite invariant measures.

\begin{enumerate}[(A)]
\item \label{3}
$\displaystyle\operatorname*{ess\,sup}_{\al\in\mb{A}}\tau'_{\al}\left(\tfrac{1}{2}\right)<\infty$;
\item \label{4}
$\displaystyle\int_{\mb{A}}\frac{1}{x_1-x_2^{\bdsy{\al}}}d\nu_{\mb{A}}(\al)<\infty$.
\end{enumerate}

Obviously the condition \eqref{3} implies the condition \eqref{4}.

\begin{lemma}\label{lem2-1}
Under the assumption \eqref{0}--\eqref{2}, the condition \eqref{4} implies the following: for any $\delta>0$, there exists $N_0\ge1$ such that
\[
\int_{\mb{A}^{\N}\times\mb{B}}\frac{y_{N_0+1}^{\bdsy{\al},\beta}-\frac{1}{2}}{x_1^{\bdsy{\al}}-x_2^{\bdsy{\al}}}d\nu_{\mb{A}}^{\infty}(\bdsy{\al})\nu_{\mb{B}}(\beta)<\delta.
\]
\end{lemma}

\begin{proof}
It follows from \eqref{1} and \eqref{2} that $y_{N+1}^{\bdsy{\al},\beta}-\frac{1}{2}\to0$ as $N\to\infty$ for $\nu_{\mb{A}}^{\infty}$-almost every $\bdsy{\al}\in\mb{A}^{\N}$ and $\nu_{\mb{B}}$-almost every $\beta\in\mb{B}$.
Then we have
\begin{linenomath}
\begin{align*}
\lim_{N\to\infty}\int_{\mb{A}^{\N}\times\mb{B}}\frac{y_{N+1}^{\bdsy{\al},\beta}-\frac{1}{2}}{x_1^{\bdsy{\al}}-x_2^{\bdsy{\al}}}d\nu_{\mb{A}}^{\infty}(\bdsy{\al})\nu_{\mb{B}}(\beta)
&=\int_{\mb{A}^{\N}\times\mb{B}}\lim_{N\to\infty}\frac{y_{N+1}^{\bdsy{\al},\beta}-\frac{1}{2}}{x_1^{\bdsy{\al}}-x_2^{\bdsy{\al}}}d\nu_{\mb{A}}^{\infty}(\bdsy{\al})\nu_{\mb{B}}(\beta)\\
&=0
\end{align*}
\end{linenomath}
from the Lebesgue dominated convergence theorem, which proves the lemma.
\end{proof}

In what follows, $\{T_{\al,\beta};\nu_{\mb{A}},\nu_{\mb{B}}:\al\in\mb{A},\ \beta\in\mb{B}\}$ denotes a random dynamical system given by the transition function \eqref{tp} with conditions \eqref{0}--\eqref{2} and is referred as a \emph{random piecewise convex map}.
In the next section, we prove the existence and uniqueness of equivalent $\sigma$-finite invariant measures for random piecewise convex maps.
Furthermore, we show the asymptotics of the invariant measures.

\section{Equivalent $\sigma$-finite invariant measures}\label{sec3}

Before stating the main results in the paper, some basic definitions are listed.
Let $\mu$ be an absolutely continuous measure with respect to $\lam$.
Recall that an \emph{invariant set} for a Markov operator $P$ is a measurable set $E\in\mathcal{B}$ with the property $P^*1_E=1_E$ $\lam$-almost everywhere and $P$ is called \emph{ergodic} with respect to $\mu$ if each invariant set $E$ satisfies either $\mu(E)=0$ or $\mu(X\setminus E)=0$.
$P$ is called \emph{conservative} with respect to $\mu$ if any function $h$ supported on $\supp\mu$ with $h\ge P^*h$ satisfies $h=P^*h$.
Other equivalent characterization of conservativity are found in \cite{Kre}*{\S3.1}.

The following main theorem establishes the existence of Lebesgue-equivalent, conservative and ergodic $\sigma$-finite invariant measures for random piecewise convex maps defined in the previous section.

\begin{theorem}\label{Thm1}
Let $\{T_{\al,\beta};\nu_{\mb{A}},\nu_{\mb{B}}:\al\in\mb{A},\ \beta\in\mb{B}\}$ be a random piecewise convex map satisfying the conditions \eqref{0}--\eqref{2} and \eqref{4} in \S \ref{sec2}.
Then, for the random piecewise convex map, there exists a conservative and ergodic $\sigma$-finite invariant measure $\mu$ which is equivalent to the Lebesgue measure $\lam$.
Moreover, the invariant density function of $\mu$, $\frac{d\mu}{d\lam}$, satisfies
\begin{enumerate}[(A)]
\setcounter{enumi}{3}
\item\label{D}
$\frac{d\mu}{d\lam}$ restricted on $(0,\frac{1}{2})$
is non-increasing $\lam$-almost everywhere,
and
\setcounter{enumi}{20}
\item\label{U}
for any $0<\vep<\frac{1}{2}$, there is a constant $C=C(\vep)>0$ such that $\frac{d\mu}{d\lam}\le C$, $\lam$-almost everywhere on $X\setminus [0,\vep)$.
\end{enumerate}
If we suppose \eqref{3} (and hence \eqref{4} is automatically fulfilled),
then it also holds
\begin{enumerate}[(A)]
\setcounter{enumi}{11}
\item\label{L}
there is a constant $c>0$ such that $\frac{d\mu}{d\lam}\ge c$, $\lam$-almost everywhere on $X$.
\end{enumerate}
\end{theorem}

\begin{remark}
(I)
Since an equivalent $\sigma$-finite measure in Theorem \ref{Thm1} is conservative and ergodic, it is unique (up to a multiplicative constant) by \cite{F0}*{Theorem A in Chapter VI}.

(II)
When we do not suppose the condition \eqref{3}, there does no longer exist lower bound for $\frac{d\mu}{d\lam}$, namely the condition \eqref{L}, in general.
See Example \ref{counter} for a counterexample.
\end{remark}

We further state two theorems of Theorem \ref{Thm1} which tell us when the invariant measure becomes an infinite measure.
The first one deals with a specific case when $\mb{A}$ is a point set, from which we can show the general case as in Theorem \ref{Cor2}.

\begin{theorem}\label{Cor1}
Let $\{T_{\al,\beta};\nu_{\mb{A}},\nu_{\mb{B}}:\al\in\mb{A},\ \beta\in\mb{B}\}$ be as in Theorem \ref{Thm1} and assume \eqref{3}.
Suppose $\mb{A}=\{\al\}$ is a singleton, and set $X_n=X_n^{\bdsy{\al}}$ and $\eta(\beta)=\eta(\bdsy{\al},\beta)$ where $\bdsy{\al}=(\al,\al,\dots)$.
Then the asymptotics of the invariant measure $\mu$ given in Theorem \ref{Thm1} is of order
\[
\mu\left(X_n\right)\approx \int_{\{\beta\in\mb{B}:\eta(\beta)<n\}}\left(y_{n-\eta(\beta)}^{\beta}-\tfrac{1}{2}\right)d\nu_{\mb{B}}(\beta)
+\nu_{\mb{B}}\left\{\beta\in\mb{B}:\eta(\beta)\ge n\right\}
\]
for $n$ large enough.
\end{theorem}

\begin{remark}\label{rem11}
In Theorem \ref{Cor1}, if $S_{\beta}$ is surjective for $\nu_{\mb{B}}$-almost every $\beta\in\mb{B}$ then by the definition of $\eta$ we have $\eta(\beta)=0$ and the invariant measure $\mu$ is of order
\[
\mu\left(X_n\right)\approx \int_{\mb{B}}\left(y_{n+1}^{\beta}-\tfrac{1}{2}\right)d\nu_{\mb{B}}(\beta).
\]
Simultaneously, the second term $\nu_{\mb{B}}\{\beta\in\mb{B}:\eta(\beta)\ge n\}$ is negligible when $\essinf_{\beta\in\mb{B}}S_{\beta}(1)>0$ and hence $\esssup_{\beta\in\mb{B}}\eta(\beta)<\infty$
(e.g., when $\#\mb{B}<\infty$). 
\end{remark}

When $\mb{A}$ is an uncountable set, the form of the invariant density is complicated in general.
However, combining Theorem \ref{Cor1} and the comparison theorem from \cite{I4}, we can estimate the size of the $\sigma$-finite invariant measure $\mu$ in Theorem \ref{Thm1} even when $\mb{A}$ is not singleton, by reducing to the case of singleton.
In order to clarify our statement, we need to introduce the following condition.
A random piecewise convex map $\{T_{\al,\beta};\nu_{\mb{A}},\nu_{\mb{B}}:\al\in\mb{A},\ \beta\in\mb{B}\}$ is said to satisfy the condition {($\dagger$)} if there are some $c\in(0,\frac{1}{2})$ and $\al_1,\al_2\in\mb{A}$ such that
\begin{linenomath}
\begin{align*}
&\nu_{\mb{A}}\left\{\al\in\mb{A}:\tau_{\al}(0,\vep)\subset\tau_{\al_1}(0,\vep) \text{ for any } \vep\in\left(0,c\right)\right\}=1 \text{ and}\\
&\nu_{\mb{A}}\left\{\al\in\mb{A}:\tau_{\al}(0,\vep)\supset\tau_{\al_2}(0,\vep) \text{ for any } \vep\in\left(0,c\right)\right\}>0.
\end{align*}
\end{linenomath}
These conditions are of course equivalent to that
\begin{linenomath}
\begin{align*}
&\nu_{\mb{A}}\left\{\al\in\mb{A}:\tau_{\al}\le\tau_{\al_1}\text{ on } \left(0,c\right)\right\}=1 \text{ and}\\
&\nu_{\mb{A}}\left\{\al\in\mb{A}:\tau_{\al}\ge\tau_{\al_2}\text{ on } \left(0,c\right)\right\}>0.
\end{align*}
\end{linenomath}

With some abuse of notation, for a fixed $\bar{\al}\in\mb{A}$, $\{T_{\bar{\al},\beta};\nu_{\mb{B}}:\beta\in\mb{B}\}$ denotes a random piecewise convex map $\{T_{\al,\beta};\nu_{\bar{\al}},\nu_{\mb{B}}:\al\in\{\bar{\al}\},\beta\in\mb{B}\}$ where $\nu_{\bar{\al}}$ is the Dirac measure on $\bar{\al}$.

\begin{theorem}\label{Cor2}
Let $\{T_{\al,\beta};\nu_{\mb{A}},\nu_{\mb{B}}:\al\in\mb{A},\ \beta\in\mb{B}\}$ and $\mu$ be as in Theorem \ref{Thm1} with the assumption \eqref{3} and satisfy the condition {$(\dagger)$} with some $\al_1,\al_2\in\mb{A}$.
Let $\mu_i$'s be $\sigma$-finite invariant measures for random piecewise convex maps $\{T_{\al_i,\beta};\nu_{\mb{B}}:\beta\in\mb{B}\}$ $(i=1,2)$ given in Theorem \ref{Cor1}.
Then there is a constant $M>0$ such that for any $a$ and $b$ with $0\le a<b\le\frac{1}{2}$ 
\[
M^{-1}\mu_1\left([a,b]\right)
\le\mu\left([a,b]\right)
\le M\mu_2\left([a,b]\right).
\]
Consequently, if $\mu_1(X)=\infty$ then $\mu(X)=\infty$, and if $\mu_2(X)<\infty$ then $\mu(X)<\infty$.
\end{theorem}

\begin{remark}
(I)
In Theorem \ref{Cor2}, $\al_1$ is chosen to be a parameter for which $\tau_{\al_1}$ dominates any other $\tau_{\al}$ for $\al\in\mb{A}$ from above and $\al_2$ should be chosen in the way that $\tau_{\al_2}$ is close to $\tau_{\al_1}$ as much as possible so that the inequality becomes sharper.
For instance, see Example \ref{subsec42} for the choice of parameters.

(II)
If $\#\mb{A}<\infty$ and $\nu_{\mb{A}}(\al)>0$ for all $\al\in\mb{A}$, then one can have $\al_1=\al_2$ in Theorem \ref{Cor2}.
Similarly, if there is a parameter $\al'\in\mb{A}$ such that $\tau_{\al}\le\tau_{\al'}$ on $(0,\frac{1}{2})$ for $\nu_{\mb{A}}$-almost every $\al\in\mb{A}$ and $\nu_{\mb{A}}(\{\al'\})>0$, then both $\al_1$ and $\al_2$ in Theorem \ref{Cor2} can be taken as $\al'$.
That is, the invariant measure $\mu$ in Theorem \ref{Cor2} is of same order of $\mu_{\al'}$, where $\mu_{\al'}$ is the invariant measure for $\{T_{\al',\beta};\nu_{\mb{B}}:\beta\in\mb{B}\}$.
\end{remark}

Before proving Theorem \ref{Thm1}, we recall the key tool, called the induced operator (or the first return map in the sense of \cite{I3}), to construct an absolutely continuous $\sigma$-finite invariant measure and we also prepare lemmas.

As in the previous section, we let $Y=[\frac{1}{2},1]$ and recall (see also \cites{F,T}) that the induced operator (on $Y$) $P_Y$ is defined by
\begin{linenomath}
\begin{align}\label{eq:io}
P_Y=I_YP\sum_{n=0}^{\infty}\left(I_{Y^c}P\right)^n
\end{align}
\end{linenomath}
where $I_Y$ and $I_{Y^c}$ are the restriction operators on $Y$ (i.e., $I_Yf=1_Yf$ for each measurable function $f$) and $Y^c$, respectively.
The operator $P_Y$ is a well-defined Markov operator over $L^1(X,\lam)$ since $Y$ is a $P$-sweep-out set with respect to $\lam$ (see Lemma 4.7 in \cite{T} precisely).
The induced operator for a Markov operator is a generalization of the induced map for a non-singular map.

For $(\bdsy{\al},\beta)\in\mb{A}^{\N}\times\mb{B}$, $\mathcal{L}_{T_Y^{(\bdsy{\al},\beta)}}$ denotes the Perron--Frobenius operator associated with the induced (random) map $T_Y^{(\bdsy{\al},\beta)}x\coloneqq \tau_{\al_1}\circ\cdots\circ\tau_{\al_{n(x)}}\circ S_{\beta}x$ where $n(x)\ge1$ is the minimum number satisfying $\tau_{\al_1}\circ\cdots\circ\tau_{\al_{n(x)}}\circ S_{\beta}x\in Y$ (such $n(x)$ exists for $x\in Y\setminus\{\frac{1}{2}\}$).

\begin{lemma}\label{lem1}
The induced operator $P_Y$ defined by the equation \eqref{eq:io} satisfies
\[
P_Yf=\int_{\mb{A}^{\N}\times\mb{B}}\mathcal{L}_{T_Y^{(\bdsy{\al},\beta)}}fd\nu_{\mb{A}}^{\infty}(\bdsy{\al})\nu_{\mb{B}}(\beta)
\]
for each $f\in L^1(Y,\lam)$ with $f=0$ $\lam$-almost everywhere on $Y^c$.
\end{lemma}

\begin{proof}
As in the equality \eqref{eq:io}, the induced operator on $Y$ and its adjoint operator are defined by
\[
P_Y = I_Y P \sum_{n=0}^\infty (I_{Y^C} P)^n 
\quad \text{and} \quad 
P^*_Y =  \sum_{n = 0}^{\infty}  \left({P^*}I_{Y^c}\right)^n  \left(P^* I_Y   \right).
\]
On the other hand, by Proposition 4.1 (iv) of \cite{I3}, 
$P^*_Y 1_A (x)$
equals to the transition function from $x$ into $A$ which defines the induced map on $Y$ of the original random map.
\end{proof}

We then prove the following key lemma.

\begin{lemma}\label{lem2}
Suppose the condition \eqref{4}.
If $f$ is non-negative, bounded and non-increasing on $Y$ and satisfies $f=0$ $\lam$-almost everywhere on $Y^c$, then so is $P_Yf$.
Moreover, if $\lV f\rV_{L^1}\le 1$ then there is some positive constant $K>0$, independent of $f$, such that for any $\delta\in(0,1)$ and $\lam$-almost every $x\in Y$,
\begin{linenomath}
\begin{align}\label{eq:key}
P_Yf(x)
<\delta f\left(\tfrac{1}{2}\right) + K.
\end{align}
\end{linenomath}
\end{lemma}

\begin{proof}
We follow the proof of Proposition 5.1 in \cite{I}.
Let $X_n^{\bdsy{\al}}\coloneqq (x_{n+1}^{\bdsy{\al}},x_n^{\bdsy{\al}}]$ and $Y_n^{\bdsy{\al},\beta}\coloneqq (y_{n+1}^{\bdsy{\al},\beta},y_n^{\bdsy{\al},\beta}]$ as before.
Then for each $(\bdsy{\al},\beta)\in\mb{A}^{\N}\times\mb{B}$, the induced map $T_Y^{(\bdsy{\al},\beta)}$ is piecewise convex such that
$T_Y^{(\bdsy{\al},\beta)}\vert_{Y_1^{\bdsy{\al},\beta}}=\tau_{\al_1}\circ\cdots\circ\tau_{\al_{\eta(\bdsy{\al},\beta)}}\circ S_{\beta}\vert_{Y_1^{\bdsy{\al},\beta}}$ maps from $Y_1^{\bdsy{\al},\beta}$ onto $[\frac{1}{2},T_Y^{(\bdsy{\al},\beta)}(1)]\subset Y$
and
$T_Y^{(\bdsy{\al},\beta)}\vert_{Y_{n+1}^{\bdsy{\al},\beta}}=\tau_{\al_1}\circ\cdots\circ\tau_{\eta(\bdsy{\al},\beta)+n}\circ S_{\beta}\vert_{Y_{n+1}^{\bdsy{\al},\beta}}$ maps from $Y_{n+1}^{\bdsy{\al},\beta}$ onto $Y$
for $n\ge1$ by construction.

If we set
\[
\vphi_1^{(\bdsy{\al},\beta)}(x) =
\begin{cases}
\dfrac{1}{\left(T_Y^{(\bdsy{\al},\beta)}\right)'\circ \left(T_Y^{(\bdsy{\al},\beta)}\big\vert_{Y_1^{\bdsy{\al},\beta}}\right)^{-1}(x)} & \left(x\in T_Y^{(\bdsy{\al},\beta)}\left(Y_1^{\bdsy{\al},\beta}\right)\right),\\
0 & \left(\text{otherwise}\right)
\end{cases}
\]
and
\[
\vphi_{n+1}^{(\bdsy{\al},\beta)}(x) = \frac{1}{\left(T_Y^{(\bdsy{\al},\beta)}\right)'\circ \left(T_Y^{(\bdsy{\al},\beta)}\big\vert_{Y_{n+1}^{\bdsy{\al},\beta}}\right)^{-1}(x)}
\quad(x\in X)
\]
for $(\bdsy{\al},\beta)\in\mb{A}^{\N}\times\mb{B}$ and $n\ge1$,
then $\vphi_n^{(\bdsy{\al},\beta)}$ is non-increasing on $Y$ for each $(\bdsy{\al},\beta)\in\mb{A}^{\N}\times\mb{B}$ and $n\ge1$.
Since for any non-negative and non-increasing function $f$ on $Y$ we have
\begin{linenomath}
\begin{align*}
P_Yf
&=
\int_{\mb{A}^{\N}\times\mb{B}} 
\Bigg(\sum_{n\ge2}\vphi_n^{(\bdsy{\al},\beta)}f \circ \left(T_Y^{(\bdsy{\al},\beta)}\big\vert_{Y_{n}^{\bdsy{\al},\beta}}\right)^{-1}\\
&\qquad\qquad\qquad+\vphi_1^{(\bdsy{\al},\beta)}f \circ \left(T_Y^{(\bdsy{\al},\beta)}\big\vert_{Y_1^{\bdsy{\al},\beta}}\right)^{-1}1_{T_Y^{(\bdsy{\al},\beta)}\left(Y_1^{\bdsy{\al},\beta}\right)}\Bigg) d\nu_{\mb{A}}^{\infty}(\bdsy{\al})\nu_{\mb{B}}(\beta)
\end{align*}
\end{linenomath}
from Lemma \ref{lem1},
$P_Yf$ is also non-negative and non-increasing and the former part of the lemma is proven.

Now from the convexity of $T_{\al,\beta}$ we can easily see that
\begin{linenomath}
\begin{align}\label{eq111}
T_{\al_{\eta(\bdsy{\al},\beta)+n},\beta}'\big\vert_{Y_n^{\bdsy{\al},\beta}}\ge\frac{\lam\left(X_{\eta(\bdsy{\al},\beta)+n}^{\bdsy{\al}}\right)}{\lam\left(Y_{n+1}^{\bdsy{\al},\beta}\right)}
\quad\text{and}\quad
T_{\al_{n},\beta}'\big\vert_{X_n^{\bdsy{\al}}}\ge\frac{\lam\left(X_n^{\bdsy{\al}}\right)}{\lam\left(X_{n+1}^{\bdsy{\al}}\right)}
\end{align}
\end{linenomath}
for any $(\bdsy{\al},\beta)\in\mb{A}^{\N}\times\mb{B}$ and $n\ge1$.
Thus it follows from \eqref{eq111} that for any $(\bdsy{\al},\beta)\in\mb{A}^{\N}\times\mb{B}$ and $n\ge 1$
\begin{linenomath}
\begin{align*}
\left(T_Y^{(\bdsy{\al},\beta)}\right)'\bigg\vert_{Y_n^{\bdsy{\al},\beta}}
&=T'_{\al_{\eta(\bdsy{\al},\beta)+n},\beta}\big\vert_{Y_n^{\bdsy{\al},\beta}}\prod_{k=1}^{\eta(\bdsy{\al},\beta)+n-1}T'_{\al_{\eta(\bdsy{\al},\beta)+n-k},\beta}\big\vert_{X_{\eta(\bdsy{\al},\beta)+n-k}}\\
&\qquad\qquad\qquad\qquad\qquad\qquad\qquad\qquad\circ T_{\al_{\eta(\bdsy{\al},\beta)+n-k+1},\beta}\circ\cdots\circ T_{\al_{\eta(\bdsy{\al},\beta)+n-1},\beta}\\
&\ge\frac{\lam\left(X_1^{\bdsy{\al}}\right)}{\lam\left(Y_{n+1}^{\bdsy{\al},\beta}\right)}.
\end{align*}
\end{linenomath}
Then it holds for each $N\ge1$ that
\begin{linenomath}
\begin{align*}
\sum_{n=N}^{\infty}\vphi_n^{(\bdsy{\al},\beta)}\left(\tfrac{1}{2}\right)
&=\sum_{n= N}^{\infty}\frac{1}{\left(T_Y^{(\bdsy{\al},\beta)}\right)'\left(y_{n+1}^{\bdsy{\al},\beta}\right)}\\
&\le\sum_{n= N}^{\infty}\frac{\lam\left(Y_{n+1}^{\bdsy{\al},\beta}\right)}{\lam\left(X_1^{\bdsy{\al}}\right)}\\
&=\frac{y_{N+1}^{\bdsy{\al},\beta}-\frac{1}{2}}{x_1^{\bdsy{\al}}-x_2^{\bdsy{\al}}}.
\end{align*}
\end{linenomath}
By Lemma \ref{lem2-1}, for any fixed $0<\delta<1$ there exists $N_0\ge1$ such that we have
\[
\int_{\mb{A}^{\N}\times\mb{B}}\sum_{n=N_0}^{\infty}\vphi_n^{(\bdsy{\al},\beta)}\left(\tfrac{1}{2}\right)d\nu_{\mb{A}}^{\infty}(\bdsy{\al})\nu_{\mb{B}}(\beta)<\delta.
\]
Since any non-increasing density function on $Y$ cannot exceed $(x-\frac{1}{2})^{-1}$ (see also \cite[Step III in Proof of Theorem 4]{LY}), it holds that, for any non-negative, bounded and non-increasing function $f$ on $Y$ with $\lV f\rV_{L^1}\le 1$,
\begin{linenomath}
\begin{align*}
P_Yf(x)&\le P_Yf\left(\tfrac{1}{2}\right)\\
&=\int_{\mb{A}^{\N}\times\mb{B}}\left( \sum_{n= N_0}^{\infty}\vphi_n^{(\bdsy{\al},\beta)}\left(\tfrac{1}{2}\right)f\left(y_{n+1}^{\bdsy{\al},\beta}\right)+\sum_{n=1}^{N_0-1}\vphi_n^{(\bdsy{\al},\beta)}\left(\tfrac{1}{2}\right)f\left(y_{n+1}^{\bdsy{\al},\beta}\right) \right)d\nu_{\mb{A}}^{\infty}(\bdsy{\al})d\nu_{\mb{B}}(\beta)\\
&<\delta f\left(\tfrac{1}{2}\right) + \int_{\mb{A}^{\N}\times\mb{B}}\sum_{n=1}^{N_0-1}\frac{\vphi_n^{(\bdsy{\al},\beta)}\left(\frac{1}{2}\right)}{y_{n+1}^{\bdsy{\al},\beta}-\tfrac{1}{2}}d\nu_{\mb{A}}^{\infty}(\bdsy{\al})d\nu_{\mb{B}}(\beta)
\end{align*}
\end{linenomath}
for $\lam$-almost every $x\in Y$.
Therefore, putting $K=\int_{\mb{A}^{\N}\times\mb{B}}\sum_{n=1}^{N_0-1}\frac{\vphi_n^{(\bdsy{\al},\beta)}(\frac{1}{2})}{y_{n+1}^{\bdsy{\al},\beta}-\frac{1}{2}}d\nu_{\mb{A}}^{\infty}(\bdsy{\al})d\nu_{\mb{B}}(\beta)<\infty$,
we have obtained the inequality \eqref{eq:key}.
\end{proof}

We now emphasize that the left branches $\tau_{\al}$'s map points surjectively onto $[0,1]$.
This together with the condition \eqref{4} guarantees that an invariant density for the induced operator $P_Y$ is fully supported on $Y$ as well as bounded above.
Furthermore, \eqref{3} ensures the invariant density to be bounded away from zero on $Y$.
Henceforth $\lam\vert_Y$ denotes the measure $\lam$ restricted on $Y$.

\begin{lemma}\label{lem3}
Under the assumption \eqref{4}, the induced operator $P_Y$ is ergodic with respect to the Lebesgue measure $\lam$ and admits a unique $\lam\vert_Y$-equivalent invariant probability measure whose density is non-increasing and bounded above on $Y$.
Moreover, if we assume \eqref{3} then the density function is bounded away from zero on $Y$.
\end{lemma}

\begin{proof}
First of all, ergodicity follows from the following argument.
For each $(\bdsy{\al},\beta)\in\mb{A}^{\N}\times\mb{B}$, the map $T_Y^{(\bdsy{\al},\beta)}$ satisfies the conditions in \cite[Proposition 5.1]{I} and is ergodic (or moreover exact) with respect to $\lam\vert_Y$.
If $D$ is an invariant set for $P_Y$, then
\[
P^*_Y1_D=\int_{\mb{A}^{\N}\times\mb{B}}1_{D}\circ T_Y^{(\bdsy{\al},\beta)}d\nu_{\mb{A}}^{\infty}(\bdsy{\al})d\nu_{\mb{B}}(\beta)=1_D.
\]
Thus $D$ is a $T_Y^{(\bdsy{\al},\beta)}$-invariant set for $\nu_{\mb{A}}^{\infty}\times\nu_{\mb{B}}$-almost everywhere.
Now it is straightforward to see that $D=\emptyset$ or $X\pmod{\lam}$ since each $T_Y^{(\bdsy{\al},\beta)}$ is ergodic.

From Lemma \ref{lem2}, $P_Y^n1_Y$ is non-increasing for $n\ge0$ and we apply \eqref{eq:key} repeatedly to get for any fixed $\delta\in(0,1)$ and for $x\in Y$
\[
P_Y^n1_Y(x)<\delta P_Y^{n-1}1_Y\left(\tfrac{1}{2}\right) + K
<\delta^2P_Y^{n-2}1_Y\left(\tfrac{1}{2}\right) + \delta K+K
\]
and so on.
Eventually, we have for $n\ge1$
\begin{linenomath}
\begin{align}\label{ineq1}
P_Y^n1_Y<\delta^n+\frac{K}{1-\delta},
\end{align}
\end{linenomath}
that is, $P_Y^n1_Y$ is bounded above by $C_0\coloneqq 1+K(1-\delta)^{-1}<\infty$ for any $n\ge1$.
Therefore, by \cite[Theorem 3.1 and Proposition 3.9]{T}, the limiting point
\[
h_0\coloneqq \frac{1}{\lam(Y)}\lim_{n\to\infty}\frac{1}{n}\sum_{i=0}^{n-1}P_Y^i1_Y
=2\lim_{n\to\infty}\frac{1}{n}\sum_{i=0}^{n-1}P_Y^i1_Y
\]
exists and is an invariant density of $P_Y$, which is conservative and ergodic implying uniqueness of the invariant density.

We then show this $h_0$ satisfies the conditions in the statement of the lemma.
From Lemma \ref{lem2} and \eqref{ineq1}, $h_0$ is non-increasing and bounded above by $C_1\coloneqq 2C_0$ on $Y$.
For the lower bound, notice that from the fact that $P_Y^n1_Y$ is non-increasing and the above inequality (\ref{ineq1})
\[
P_Y^n1_Y\ge 1
\]
on $[\frac{1}{2},\frac{1}{2}+\frac{1}{C_1}]$ for $n\ge 1$ so that
\begin{linenomath}
\begin{align}\label{112}
h_0\ge 2 \quad\text{on } \left[\tfrac{1}{2},\tfrac{1}{2}+\tfrac{1}{C_1}\right].
\end{align}
\end{linenomath}
On the other hand, it follows from the Lebsgue dominated convergence theorem (see also the proof to Lemma \ref{lem2-1}), we get $N_0\ge2$ such that
\[
\int_{\mb{A}^{\N}\times\mb{B}}\left(y_{N_0}^{\bdsy{\al},\beta}-\tfrac{1}{2}\right)d\nu_{\mb{A}}^{\infty}(\bdsy{\al})d\nu_{\mb{B}}(\beta)
<\frac{1}{C_1}.
\]
We define
\[
E\coloneqq\left\{(\bdsy{\al},\beta)\in\mb{A}^{\N}\times\mb{B}:\bigcup_{n=N_0}^{\infty}Y_n^{\bdsy{\al},\beta}\subset\left[\frac{1}{2},\frac{1}{2}+\frac{1}{C_1}\right]\right\}.
\]
Then since it holds that
\[
y_{N_0}^{\bdsy{\al},\beta}-\frac{1}{2}=\lam\left(\bigcup_{n=N_0}^{\infty}Y_n^{\bdsy{\al},\beta}\right),
\]
we have $\nu_{\mb{A}}^{\infty}\times\nu_{\mb{B}}(E)>0$.
Combining \eqref{112} with the above argument, we have
\[
h_0\ge2\sum_{n=N_0}^{\infty}1_{Y_n^{\bdsy{\al},\beta}}
\quad\text{for any }(\bdsy{\al},\beta)\in E.
\]

Thus, taking $\vphi_n$ defined in the proof to Lemma \ref{lem2} into account, we have by Lemma \ref{lem1}
\begin{linenomath}
\begin{align*}
h_0&=P_Y h_0=
\int_{\mb{A}^{\N}\times\mb{B}} \mathcal{L}_{T_{Y}^{(\bdsy{\al},\beta)}}h_0 d\nu_{\mb{A}}^{\infty}(\bdsy{\al})d\nu_{\mb{B}}(\beta)\\
&\ge2\int_{E} \sum_{n=N_0}^{\infty}\mathcal{L}_{T_{Y}^{(\bdsy{\al},\beta)}}1_{Y_n^{(\bdsy{\al},\beta)}} d\nu_{\mb{A}}^{\infty}(\bdsy{\al})d\nu_{\mb{B}}(\beta)\\
&= 2\int_{E}\sum_{n= N_0}^{\infty}\vphi_n^{(\bdsy{\al},\beta)}d\nu_{\mb{A}}^{\infty}(\bdsy{\al})d\nu_{\mb{B}}(\beta)
\end{align*}
\end{linenomath}
on $Y$ since $T_{Y}^{(\bdsy{\al},\beta)}\vert_{Y_n^{(\bdsy{\al},\beta)}}$ is surjective for each $n\ge2$.
The conditions \eqref{1} and \eqref{2} imply that
\[
C_2 (x)\coloneqq 2\int_{E}\sum_{n= N_0}^{\infty}\vphi_n^{(\bdsy{\al},\beta)}(x)d\nu_{\mb{A}}^{\infty}(\bdsy{\al})d\nu_{\mb{B}}(\beta)>0
\]
$\lam\vert_Y$-almost every $x\in Y$.
Therefore, we conclude $C_1\ge h_0(x)\ge C_2(x)$ on $Y$.
Moreover, if we assume \eqref{3} then $\essinf_{x\in Y} C_2(x)>0$.
The proof is completed.
\end{proof}

\begin{proof}[Proof of Theorem \ref{Thm1}]
The well-known formula of invariant measures via the induced operators (see Proposition 4.14 in \cite{T} for example) shows that
\[
h=\sum_{n=0}^{\infty}(I_{Y^c}P)^nh_0
=h_0+I_{Y^c}\sum_{n=0}^{\infty}(PI_{Y^c})^nPh_0.
\]
gives an invariant density function of an absolutely continuous $\sigma$-finite invariant measure $\mu$ for $P$ where $h_0$ is the invariant density of $P_Y$ obtained in Lemma \ref{lem3}.
Then it follows from the fact that $h_0$ is supported on $Y$ that
\begin{linenomath}
\begin{align}\label{eqform}
h=h_0+\int_{\mb{A}^{\N}\times\mb{B}}\sum_{n=0}^{\infty}I_{Y^c}\mathcal{L}_{\tau_{\al_0}}\cdots\mathcal{L}_{\tau_{\al_{n-1}}}\mathcal{L}_{S_{\beta}}h_0 d\nu_{\mb{A}}^{\infty}(\al_0,\al_1,\dots)d\nu_{\mb{B}}(\beta).
\end{align}
\end{linenomath}
Since $\tau_{\al}$ for each $\al\in\mb{A}$ is surjective and the support of $h_0$ is $Y$ up to $\lam$-measure zero sets, $h$ is evidently fully supported on $X$ and thus the invariant measure is equivalent to $\lam$.
Now that $h_0$ is non-increasing and so is $(PI_{Y^c})^nh_0$ for $n\ge0$ from the similar argument of the proof to Lemma \ref{lem2} together with the assumption \eqref{2}, we have \eqref{D}.
Then \eqref{U} follows from \eqref{D} and the fact that $\mu$ is $\sigma$-finite.

If we assume \eqref{3}, then we have $C^{-1}1_Y\le h_0\le C1_Y$ for some $C\ge1$ by Lemma \ref{lem3}.
Or equivalently, for each $\bdsy{\al}\in\mb{A}^{\N}$ and $\beta\in\mb{B}$
\[
C^{-1}\sum_{n=1}^{\infty}1_{Y_n^{\bdsy{\al},\beta}}\le h_0 \le C\sum_{n=1}^{\infty}1_{Y_n^{\bdsy{\al},\beta}}.
\]
Note that for the bound above (or the desired consequence \eqref{U}), we only need the condition \eqref{4}.

We first observe that for $n\ge2$ and $\beta\in\mb{B}$,
\begin{linenomath}
\begin{align*}
\mathcal{L}_{T_{\al,\beta}}1_{Y_n^{\bdsy{\al},\beta}}
&=\frac{1_{Y_n^{\bdsy{\al},\beta}}\circ S_{\beta}\vert_{Y_n^{\bdsy{\al},\beta}}^{-1}}{S_{\beta}'\circ S_{\beta}\vert_{Y_n^{\bdsy{\al},\beta}}^{-1}}
\le\dfrac{\lam\left(Y_{n+1}^{\bdsy{\al},\beta}\right)}{\lam\left(X_{\eta(\bdsy{\al},\beta)+n}\right)}1_{X_{\eta(\bdsy{\al},\beta)+n-1}},\text{ and}\\
\mathcal{L}_{T_{\al,\beta}}1_{Y_n^{\bdsy{\al},\beta}}&\ge\dfrac{\lam\left(Y_{n-1}^{\bdsy{\al},\beta}\right)}{\lam\left(X_{\eta(\bdsy{\al},\beta)+n-2}\right)}1_{X_{\eta(\bdsy{\al},\beta)+n-1}}
\end{align*}
\end{linenomath}
by the convexity of $S_{\beta}$.
Hence, taking it into account that
\[
\bigcup_{n=\eta(\bdsy{\al},\beta)+1}^{\infty}X_n^{\bdsy{\al}}\subset S_{\beta}Y\subset\bigcup_{n=\eta(\bdsy{\al},\beta)}^{\infty}X_n^{\bdsy{\al}},
\]
for each $\bdsy{\al}\in\mb{A}^{\N}$ we have
\begin{linenomath}
\begin{align}
Ph_0
&=\int_{\mb{B}}\mathcal{L}_{T_{\al,\beta}}h_0d\nu_{\mb{B}}(\beta)
\le C\int_{\mb{B}}\sum_{n=1}^{\infty}\frac{\lam\left(Y_{n+1}^{\bdsy{\al},\beta}\right)}{\lam\left(X_{\eta(\bdsy{\al},\beta)+n}^{\bdsy{\al}}\right)}1_{X_{\eta(\bdsy{\al},\beta)+n-1}^{\bdsy{\al}}}d\nu_{\mb{B}}(\beta),\text{ and} \label{ineq2}\\
Ph_0&\ge C^{-1}\int_{\mb{B}}\sum_{n=2}^{\infty}\frac{\lam\left(Y_{n-1}^{\bdsy{\al},\beta}\right)}{\lam\left(X_{\eta(\bdsy{\al},\beta)+n-2}^{\bdsy{\al}}\right)}1_{X_{\eta(\bdsy{\al},\beta)+n-1}^{\bdsy{\al}}}d\nu_{\mb{B}}(\beta) \label{ineq222}
\end{align}
\end{linenomath}
where we define $X_0^{\bdsy{\al}}\coloneqq Y$.
Note that for each $n\ge2$, $\tau_{\bdsy{\al}}^{k}\coloneqq\tau_{\al_{n-k+1}}\circ\tau_{\al_{n-k+2}}\circ\cdots\circ\tau_{\al_n}\vert_{X_n^{\bdsy{\al}}}:X_n^{\bdsy{\al}}\to X_{n-k}^{\bdsy{\al}}$ is convex for $1\le k\le n-1$ where $\bdsy{\al}=(\al_1,\al_2,\dots)$.
Thus we have
\begin{linenomath}
\begin{align}
\frac{\lam\left(X_{n-1}^{\bdsy{\al}}\right)}{\lam\left(X_{n-1-k}^{\bdsy{\al}}\right)}1_{X_{n-k}^{\bdsy{\al}}}
\le\mathcal{L}_{\tau_{\bdsy{\al}}^k}1_{X_n^{\bdsy{\al}}}
\le\frac{\lam\left(X_{n+1}^{\bdsy{\al}}\right)}{\lam\left(X_{n+1-k}^{\bdsy{\al}}\right)}1_{X_{n-k}^{\bdsy{\al}}} \label{ineq11}
\end{align}
\end{linenomath}
for $1\le k\le n-2$ and
\begin{linenomath}
\begin{align}
\frac{1}{\tau_{\al_1}'\left(\frac{1}{2}\right)}\cdot\frac{\lam\left(X_{n-1}^{\bdsy{\al}}\right)}{\lam\left(X_{1}^{\bdsy{\al}}\right)}1_{X_1^{\bdsy{\al}}}
\le\mathcal{L}_{\tau_{\bdsy{\al}}^{n-1}}1_{X_n^{\bdsy{\al}}}
\le\frac{\lam\left(X_{n+1}^{\bdsy{\al}}\right)}{\lam\left(X_{1}^{\bdsy{\al}}\right)}1_{X_1^{\bdsy{\al}}}. \label{ineq33}
\end{align}
\end{linenomath}
Note that $h-h_0$ is supported on $Y^c=\bigcup_{n=1}^{\infty}X_n^{\bdsy{\al}}\pmod{\lam}$ for any $\bdsy{\al}\in\mb{A}^{\N}$, where $h$ is a $P$-invariant locally integrable function given in \eqref{eqform}.
Then by combining the inequality \eqref{ineq2} and \eqref{ineq222} with \eqref{ineq11} and \eqref{ineq33},
it also follows from  $I_{Y^c}\mathcal{L}_{\tau_{\al}1_{X_1^{\bdsy{\al}}}}=0$ for any $\al\in\mb{A}$ that for each $N\ge1$
\begin{linenomath}
\begin{align*}
h-h_0
&\le C\int_{\mb{A}^{\N}\times\mb{B}}{\sum_{n= 1+\delta_{0,\eta(\bdsy{\al},\beta)}}^{\infty}}\sum_{k=0}^{\eta(\bdsy{\al},\beta)+n-2}\frac{\lam\left(Y_{n+1}^{\bdsy{\al},\beta}\right)}{\lam\left(X_{\eta(\bdsy{\al},\beta)+n-k}^{\bdsy{\al}}\right)}1_{X_{\eta(\bdsy{\al},\beta)+n-k-1}^{\bdsy{\al}}}d\nu_{\mb{A}}^{\infty}(\bdsy{\al})d\nu_{\mb{B}}(\beta),\text{ and}\\
h-h_0
&\ge C^{-1}\int_{\mb{A}^{\N}\times\mb{B}}{\sum_{n= 2}^{\infty}}\sum_{k=0}^{\eta(\bdsy{\al},\beta)+n-2}\frac{1}{\tau_{\al_1}'\left(\frac{1}{2}\right)}\cdot\frac{\lam\left(Y_{n-1}^{\bdsy{\al},\beta}\right)}{\lam\left(X_{\eta(\bdsy{\al},\beta)+n-k-1}^{\bdsy{\al}}\right)}1_{X_{\eta(\bdsy{\al},\beta)+n-k-1}^{\bdsy{\al}}}d\nu_{\mb{A}}^{\infty}(\bdsy{\al})d\nu_{\mb{B}}(\beta)
\end{align*}
\end{linenomath}
where
$\delta_{0,\eta(\bdsy{\al},\beta)}$ is the Dirac delta function.
Here $n$ for the summand of the upper bound for $h-h_0$ runs from $1+\delta_{0,\eta(\bdsy{\al},\beta)}$ in order that the union of $X_{\eta(\bdsy{\al},\beta)+n-k-1}^{\bdsy{\al}}$'s coincides with $Y^c$.
Comparing the coefficients of $1_{X_m}$ above, we have
\begin{linenomath}
\begin{align*}
h-h_0
&\le C\int_{\mb{A}^{\N}\times\mb{B}}\sum_{m=1}^{\infty}\underset{\eta(\bdsy{\al},\beta)+n\ge m+1}{\sum_{n\ge 1+\delta_{0,\eta(\bdsy{\al},\beta)}}}\frac{\lam\left(Y_{n+1}^{\bdsy{\al},\beta}\right)}{\lam\left(X_{m+1}^{\bdsy{\al}}\right)}1_{X_m^{\bdsy{\al}}}d\nu_{\mb{A}}^{\infty}(\bdsy{\al})d\nu_{\mb{B}}(\beta),\text{ and}\\
h-h_0
&\ge C^{-1}\int_{\mb{A}^{\N}\times\mb{B}}\sum_{m=1}^{\infty}\underset{\eta(\bdsy{\al},\beta)+n\ge m}{\sum_{n\ge 2}}\frac{1}{\tau_{\al_1}'\left(\frac{1}{2}\right)}\cdot\frac{\lam\left(Y_{n-1}^{\bdsy{\al},\beta}\right)}{\lam\left(X_m^{\bdsy{\al}}\right)}1_{X_m^{\bdsy{\al}}}d\nu_{\mb{A}}^{\infty}(\bdsy{\al})d\nu_{\mb{B}}(\beta).
\end{align*}
\end{linenomath}
For fixed $m\ge1$, we have that
\[
\underset{\eta(\bdsy{\al},\beta)+n\ge m}{\sum_{n\ge 2}}\frac{\lam\left(Y_{n-1}^{\bdsy{\al},\beta}\right)}{\lam\left(X_m^{\bdsy{\al}}\right)}
=\sum_{n=\max\{2,m-\eta(\bdsy{\al},\beta)\}}^{\infty} \frac{\lam\left(Y_{n-1}^{\bdsy{\al},\beta}\right)}{\lam\left(X_m^{\bdsy{\al}}\right)}
=\frac{y_{\max\{1,m-\eta(\bdsy{\al},\beta)-1\}}^{\bdsy{\al},\beta}-\frac{1}{2}}{\lam\left(X_m^{\bdsy{\al}}\right)}.
\]
Note that
\[
y_{\max\{1,m-\eta(\bdsy{\al},\beta)-1\}}^{\bdsy{\al},\beta}
=
\begin{cases}
1 & \text{if }\eta(\bdsy{\al},\beta)+2>m,\\
S_{\beta}^{-1}\left(x_{m-1}^{\bdsy{\al}}\right) & \text{if }m\ge2+\eta(\bdsy{\al},\beta).
\end{cases}
\]
If $\eta(\bdsy{\al},\beta)+2>m$ then
\[
\underset{\eta(\bdsy{\al},\beta)+n\ge m}{\sum_{n\ge 2}}\frac{\lam\left(Y_{n-1}^{\bdsy{\al},\beta}\right)}{\lam\left(X_m^{\bdsy{\al}}\right)}
=\frac{1}{2\lam\left(X_m^{\bdsy{\al}}\right)}
\ge\frac{1}{2}
\]
and if $m\ge\eta(\bdsy{\al},\beta)+2$ then
\begin{linenomath}
\begin{align*}
\underset{\eta(\bdsy{\al},\beta)+n\ge m}{\sum_{n\ge 2}}\frac{\lam\left(Y_{n-1}^{\bdsy{\al},\beta}\right)}{\lam\left(X_m^{\bdsy{\al}}\right)}
&=\frac{S_{\beta}^{-1}\left(x_{m-1}^{\bdsy{\al}}\right)-S_{\beta}^{-1}(0)}{\lam\left(X_m^{\bdsy{\al}}\right)}\\
&\ge\dfrac{x_{m-1}^{\bdsy{\al}}-0}{2\left(x_m^{\bdsy{\al}}-x_{m+1}^{\bdsy{\al}}\right)}
\ge\dfrac{x_{m-1}^{\bdsy{\al}}}{2x_m^{\bdsy{\al}}}\ge\dfrac{1}{2}
\end{align*}
\end{linenomath}
again from the convexity of $S_{\beta}$.
Under the assumption \eqref{3}, we also have a lower bound $\frac{1}{2C}$ for $h$.
Therefore, we conclude $h$ is bounded above on the complement of each small neighborhood of $0$ and, under the assumption \eqref{3}, bounded away from zero on $X$ as well.

The conservativity of $\mu$ follows from \cite[Remark 12]{T} and ergodicity follows from Lemma \ref{lem3} and \cite[Proposition 2.1]{T2}.
\end{proof}

\begin{proof}[Proof of Theorem \ref{Cor1}]
In this proof, since $\mb{A}$ is a singleton, we write $X_n$ and $Y_n^{\beta}$ instead of $X_n^{\bdsy{\al}}$ and $Y_n^{\bdsy{\al},\beta}$.
As shown in the proof of Theorem \ref{Thm1}, the invariant density function of $\mu$ satisfies
\begin{linenomath}
\begin{align*}
h-h_0
&\le C\int_{\mb{B}}\sum_{m=1}^{\infty}\underset{\eta(\beta)+n\ge m+1}{\sum_{n\ge 1+\delta_{0,\eta(\beta)}}}\frac{\lam\left(Y_{n+1}^{\beta}\right)}{\lam\left(X_{m+1}\right)}1_{X_m} d\nu_{\mb{B}}(\beta),\text{ and}\\
h-h_0
&\ge C^{-1}\int_{\mb{B}}\sum_{m=1}^{\infty}\underset{\eta(\beta)+n\ge m}{\sum_{n\ge 2}}\frac{\lam\left(Y_{n-1}^{\beta}\right)}{\lam\left(X_m\right)}1_{X_m} d\nu_{\mb{B}}(\beta)
\end{align*}
\end{linenomath}
for some $C>0$.
Therefore, integrating the above inequalities over $X_m$, for $m\ge1$ large enough, we have
\begin{linenomath}
\begin{align*}
\mu(X_m)&\gtrapprox\int_{\mb{B}}\underset{\eta(\beta)+n\ge m}{\sum_{n\ge2}}\lam\left(Y_{n-1}^{\beta}\right)d\nu_{\mb{B}}(\beta)\\
&\gtrapprox\int_{\left\{\beta\in\mb{B}:\eta(\beta)<m\right\}}\left(y_{m-\eta(\beta)}^{\beta}-\tfrac{1}{2}\right)d\nu_{\mb{B}}(\beta)+\nu_{\mb{B}}\left\{\beta\in\mb{B}:\eta(\beta)\ge m\right\}.
\end{align*}
\end{linenomath}
The upper estimate of asymptotics of $\mu(X_m)$ is almost same and omitted.
The proof is completed.
\end{proof}

\begin{proof}[Proof of Theorem \ref{Cor2}]
Since the density functions of $\mu$, $\mu_1$ and $\mu_2$ restricted on $Y$ are all bounded above and away from zero, the assumptions of comparison theorems (Theorem 6.2 and Theorem 6.5) in \cite{I4} are fulfilled.
\end{proof}

\section{Examples}\label{sec4}

In this section, we apply our result to several random piecewsise convex maps.

\subsection{Random piecewise linear maps with low slopes}\label{subsec41}

Let $\mb{B}\subset\N$ and $p_{\beta}\coloneqq\nu_{\mb{B}}(\{\beta\})$ a point mass on $\mb{B}$.
We define for $\beta\in\mb{B}$
\begin{linenomath}
\begin{align}\label{ex-1}
T_{\beta}x=
\begin{cases}
2x &x\in\left[0,\frac{1}{2}\right],\\
2^{-\beta}(2x-1) &x\in\left(\frac{1}{2},1\right].
\end{cases}
\end{align}
\end{linenomath}
This obviously satisfies \eqref{0}--\eqref{2} and \eqref{3}.
Note that the left branch of $T_{\beta}x=2x$ does not vary at all, and $\mb{A}$ is interpreted as a singleton.
By the definition of $T_{\beta}$, $x_n=\frac{1}{2^n}$ and $X_n=(\tfrac{1}{2^{n+1}},\tfrac{1}{2^{n}}]$ for $n\ge1$.
Thus we have $\eta(\beta)=\beta$ and
\[
y_{n+1}^{\beta}=\frac{1}{2}+\frac{1}{2^{n+1}}.
\]
Then we can apply Theorem \ref{Thm1} and Theorem \ref{Cor1} to get

\begin{proposition}\label{prop--1}
The random piecewise convex map given by \eqref{ex-1} admits a $\lam$-equivalent, conservative and ergodic $\sigma$-finite invariant measure $\mu$ such that
\[
\mu\left(\left(\tfrac{1}{2^{n+1}},\tfrac{1}{2^{n}}\right]\right)
\approx\frac{\sum_{\beta\in\mb{B}:\beta<n}2^{\beta}p_{\beta}}{2^n}+\sum_{\beta\in\mb{B}:\beta\ge n}p_{\beta}
\]
for $n$ large enough.
\end{proposition}

\begin{remark}
By the above proposition, for example, when $\mb{B}=\{[k^n]:n\ge1\}$ and $p_{[k^n]}=2^{-n}$ for some $k\ge2$, $\mu$ is infinite if $k\ge2$, where $[\,\cdot\,]$ denotes its integer part.
Indeed, in this setting, we can write
\begin{linenomath}
\begin{align*}
\mu\left(\left(\tfrac{1}{2^{n+1}},\tfrac{1}{2^{n}}\right]\right)
&\approx 2^{-n}\sum_{1\le s<\log_kn}2^{[k^s]}2^{-s}+\sum_{s\ge\log_kn}2^{-s}\\
&\gtrapprox n^{-1}
\end{align*}
\end{linenomath}
and this shows the desired conclusion.
\end{remark}

\subsection{Random weakly expanding map with positive derivative}\label{subsec42}
We first note that this example contains random LSV maps.
Let $\mb{A}=[\al_0,\al_1]$ for some $0<\al_0<\al_1<\infty$ and $\mb{B}$ be some parameter space.
We set probability measures $\nu_{\mb{A}}$ and $\nu_{\mb{B}}$ on the parameter spaces $\mb{A}$ and $\mb{B}$ respectively.
For $\al\in\mb{A}$ and $\beta\in\mb{B}$, we define
\begin{linenomath}
\begin{align}\label{ex0}
T_{\al,\beta}x=
\begin{cases}
x\left(1+2^{\al}x^{\al}\right) &x\in\left[0,\frac{1}{2}\right],\\
S_{\beta}x &x\in\left(\frac{1}{2},1\right]
\end{cases}
\end{align}
\end{linenomath}
where we assume the conditions \eqref{0}--\eqref{2}.
The condition \eqref{3} holds since $\al_1<\infty$.
Suppose further that there exists $\gamma>0$ such that for $\nu_{\mb{B}}$-almost every $\beta\in\mb{B}$ it holds $S_{\beta}'x>\gamma$.
This implies $\esssup_{\mb{B}}\eta(\beta)<\infty$.
Moreover, since $\gamma(x-\frac{1}{2})\le T\vert_{(\frac{1}{2},1]}\le 2(x-\frac{1}{2})$ by the convexity of $S_{\beta}$, for each $\bdsy{\al}=(\al,\al,\dots)\in\mb{A}^{\N}$ we have
\[
\frac{x_n^{\bdsy{\al}}}{2}\le y_{n+1}^{\bdsy{\al},\beta}-\frac{1}{2}\le \frac{x_n^{\bdsy{\al}}}{\gamma}
\]
for large $n\ge 1$.
According to the asymptotic approximation \eqref{eq1}, we have
\[
y_{n+1}^{\bdsy{\al},\beta}-\frac{1}{2}\approx n^{-\frac{1}{\al}}
\]
for each $\bdsy{\al}=(\al,\al,\dots)\in\mb{A}^{\N}$ and $\beta\in\mb{B}$.
Note that if $\bar{\al}\ge\al$ then $T_{\al,\beta}(x)\ge T_{\bar{\al},\beta}(x)$ for any $x\in[0,\frac{1}{2}]$ and $\beta\in\mb{B}$.
Then applying Theorem \ref{Cor2} to this model, we have the following.

\begin{proposition}\label{prop0}
The random piecewise convex map derived from \eqref{ex0} admits a $\lam$-equivalent, conservative and ergodic $\sigma$-finite invariant measure $\mu$ such that for any $\al'\in\mb{A}$ with $\nu_{\mb{A}}\{\al\in\mb{A}:\al'\ge\al\}>0$
\[
n^{-\frac{1}{\al_0}}\lessapprox\mu\left(X_n^{\bdsy{\al_0}}\right)
\text{ and }
\mu\left(X_n^{\bdsy{\al'}}\right)\lessapprox n^{-\frac{1}{\al'}}
\]
for $n$ large enough, where $\bdsy{\al}_0=(\al_0,\al_0,\dots)$ and $\bdsy{\al'}=(\al',\al',\dots)$.
\end{proposition}

As consequences of Proposition \ref{prop0}, $\mu([0,1])=\infty$ if $\al_0\ge1$.
Also if there is some $\al'\in[\al_0,1)$ such that $\nu_{\mb{A}}\{\al\in\mb{A}:\al'\ge\al\}>0$ then $\mu([0,1])<\infty$.

\subsection{Random weakly expanding maps with uniformly contracting branches}\label{subsec43}
Let $\mb{A}=[\al_0,\al_1]$ for some $0<\al_0<\al_1<\infty$ and $\mb{B}=[0,1]$.
We set probability measures $\nu_{\mb{A}}$ and $\nu_{\mb{B}}$ on the parameter spaces $\mb{A}$ and $\mb{B}$ respectively.
For $\al\in\mb{A}$ and $\beta\in\mb{B}$, we define
\begin{linenomath}
\begin{align}\label{ex01}
T_{\al,\beta}x=
\begin{cases}
x\left(1+2^{\al}x^{\al}\right) &x\in\left[0,\frac{1}{2}\right],\\
\beta\left(x-\frac{1}{2}\right) &x\in\left(\frac{1}{2},1\right].
\end{cases}
\end{align}
\end{linenomath}
Then \eqref{0}--\eqref{2} and \eqref{3} are satisfied.
If we set $\mb{B}_k(\bdsy{\al})\coloneqq\{\beta\in\mb{B}:\eta(\bdsy{\al},\beta)=k\}$ for $\bdsy{\al}=(\al,\al,\dots)\in\mb{A}^{\N}$, then $\mb{B}=\bigcup_{k=1}^{\infty}\mb{B}_k(\bdsy{\al})$ (disjoint) for each $\bdsy{\al}\in\mb{A}^{\N}$.
For $\beta\in\mb{B}_k(\bdsy{\al})$, by \eqref{LSV}, we have
\[
y_{n+1}^{\bdsy{\al},\beta}-\frac{1}{2}\approx\frac{(n+k)^{-\frac{1}{\al}}}{\beta}
\quad\text{and}\quad
y_{n-\eta(\bdsy{\al},\beta)+1}^{\bdsy{\al},\beta}-\frac{1}{2}\approx\frac{n^{-\frac{1}{\al}}}{\beta}.
\]
Hence, Theorem \ref{Cor1} ensures a random piecewise convex map $\{T_{\al,\beta};\nu_{\mb{B}}:\beta\in\mb{B}\}$, where $\al\in\mb{A}$ is fixed with $\bdsy{\al}=(\al,\al,\dots)\in\mb{A}^{\N}$, to have an invariant measure $\mu_{\bdsy{\al}}$ such that
\begin{linenomath}
\begin{align*}
\mu_{\bdsy{\al}}\left(X_n^{\bdsy{\al}}\right)
&\approx\sum_{k=1}^{n-1}\int_{\{\beta:\eta(\beta)<n\}\cap\mb{B}_k(\bdsy{\al})}\frac{n^{-\frac{1}{\al}}}{\beta}d\nu_{\mb{B}}(\beta)
+\nu_{\mb{B}}\left\{\beta\in\mb{B}:\eta(\beta)\ge n\right\}\\
&\approx n^{-\frac{1}{\al}} \sum_{k=1}^{n-1}\int_{\mb{B}_k(\bdsy{\al})} \frac{1}{\beta}d\nu_{\mb{B}}(\beta)
+\sum_{k=n}^{\infty}\nu_{\mb{B}}\left(\mb{B}_k(\bdsy{\al})\right)\\
\end{align*}
\end{linenomath}
for $n$ large enough.
Thus we have

\begin{proposition}\label{prop444}
The random piecewise convex map derived from \eqref{ex01} admits a $\lam$-equivalent, conservative and ergodic $\sigma$-finite invariant measure $\mu$ such that for any $\al'\in\mb{A}$ with $\nu_{\mb{A}}\{\al\in\mb{A}:\al'\ge\al\}>0$
\begin{linenomath}
\begin{align*}
&n^{-\frac{1}{\al_0}}\sum_{k=1}^{n-1}\int_{\mb{B}_k(\bdsy{\al}_0)} \frac{1}{\beta}d\nu_{\mb{B}}(\beta)
+\sum_{k=n}^{\infty}\nu_{\mb{B}}\left(\mb{B}_k(\bdsy{\al}_0)\right)
\lessapprox\mu\left(X_n^{\bdsy{\al_0}}\right)
\quad\text{and}\\
&\qquad\qquad\qquad\mu\left(X_n^{\bdsy{\al'}}\right)\lessapprox
n^{-\frac{1}{\al'}}\sum_{k=1}^{n-1}\int_{\mb{B}_k(\bdsy{\al'})} \frac{1}{\beta}d\nu_{\mb{B}}(\beta)
+\sum_{k=n}^{\infty}\nu_{\mb{B}}\left(\mb{B}_k(\bdsy{\al'})\right)
\end{align*}
\end{linenomath}
for $n$ large enough, where $\bdsy{\al}_0=(\al_0,\al_0,\dots)$ and $\bdsy{\al'}=(\al',\al',\dots)$.
\end{proposition}

\begin{remark}\label{rem42}
As an example, let $\nu_{\mb{A}}$ be the normalized Lebesgue measure on $\mb{A}=[\frac{1}{2},2]$ (that is, $\al_0=\frac{1}{2}$, $\al_1=2$ and $\al'$ can be taken an arbitrary number in $(\al_0,\al_1]$) and $d\nu_{\mb{B}}(\beta)=(1-\ell)\beta^{-\ell}$ on $\mb{B}=[0,1]$ for some $\ell\in(0,1)$.
Then Proposition \ref{prop444} tells us that $\mu([0,1])=\infty$ if $\ell>\frac{1}{2}$.
Indeed, for each $\bdsy{\al}=(\al,\al,\dots)$ where $\al\in(\frac{1}{2},2]$, since $\sup_{\mb{B}_{k+1}(\bdsy{\al})}\beta=\inf_{\mb{B}_k(\bdsy{\al})}\beta$ for $k\ge1$ and $\sup_{\mb{B}_1(\bdsy{\al})}\beta=1$, we have
\begin{linenomath}
\begin{align*}
n^{-\frac{1}{\al}}\sum_{k=1}^{n-1}\int_{\mb{B}_k(\bdsy{\al})} \frac{1}{\beta}d\nu_{\mb{B}}(\beta)
&=n^{-\frac{1}{\al}}\int_{\inf_{\mb{B}_{n-1}(\bdsy{\al})}\beta}^{1} \frac{1-\ell}{\beta^{1+\ell}}d\beta\\
&\approx n^{-\frac{1}{\al}}\left(\sup_{\beta\in\mb{B}_{n}(\bdsy{\al})}\beta^{-\ell}-1\right)\\
&\approx n^{-\frac{1-\ell}{\al}}
\end{align*}
\end{linenomath}
and
\begin{linenomath}
\begin{align*}
\sum_{k=n}^{\infty}\nu_{\mb{B}}\left(\mb{B}_k(\bdsy{\al})\right)
&=\sum_{k=n}^{\infty}\int_{\mb{B}_k(\bdsy{\al})}\frac{1-\ell}{\beta^{\ell}}d\beta
=\sum_{k=n}^{\infty}\left(\sup_{\beta\in\mb{B}_k(\bdsy{\al})}\beta^{1-\ell}-\inf_{\beta\in\mb{B}_k(\bdsy{\al})}\beta^{1-\ell}\right)\\
&=\sup_{\beta\in\mb{B}_n(\bdsy{\al})}\beta^{1-\ell}\\
&\approx n^{-\frac{1-\ell}{\al}}
\end{align*}
\end{linenomath}
for large $n$.
Here we used that any point $\beta$ in $\mb{B}_k(\bdsy{\al})$ can be approximated by
\[
(k+1)^{-\frac{1}{\al}}\lessapprox\beta\lessapprox k^{-\frac{1}{\al}}
\]
asymptotically.
Then, since we can choose $\al$ arbitrary close to $\frac{1}{2}$, the claim follows.
\end{remark}

\subsection{Random weakly expanding maps with a critical point}\label{subsec44}

Let $\mb{A}\subset(0,+\infty)$ and $\mb{B}\subset(1,+\infty)$ be compact sets and $\nu_{\mb{A}}$ and $\nu_{\mb{B}}$ be probability measures on $\mb{A}$ and $\mb{B}$, respectively.
We let $\al_0\coloneqq\min_{\mb{A}}\al$.
For $\al\in\mb{A}$ and $\beta\in\mb{B}$, define
\begin{linenomath}
\begin{align}\label{ex1}
T_{\al,\beta}x=
\begin{cases}
x\left(1+2^{\al}x^{\al}\right) &x\in\left[0,\frac{1}{2}\right],\\
2^{\beta}\left(x-\frac{1}{2}\right)^{\beta} &x\in\left(\frac{1}{2},1\right].
\end{cases}
\end{align}
\end{linenomath}

\begin{figure}[h]
\centering
\includegraphics[width=8cm]{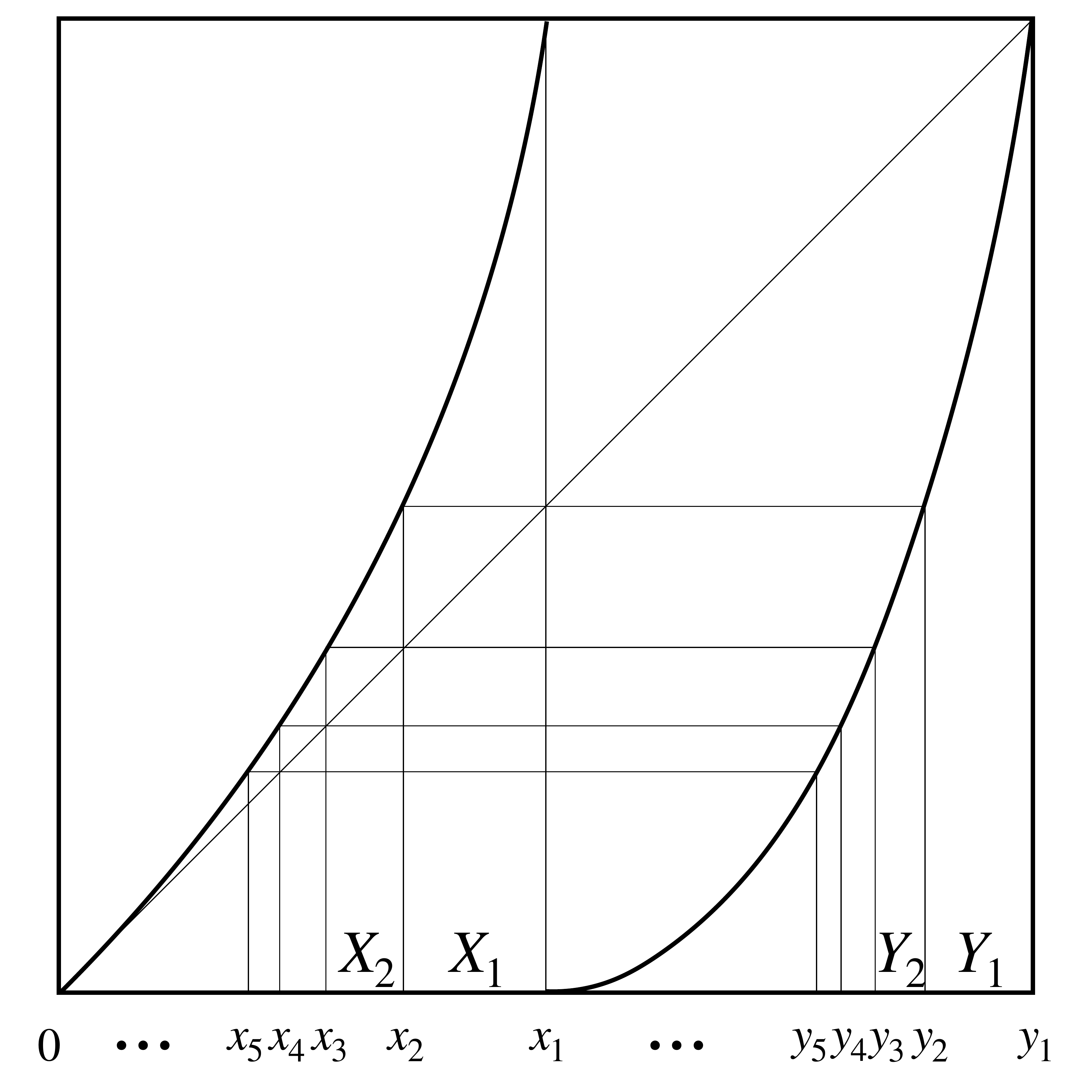}
\caption{The graph of $T_{\al,\beta}$ from \eqref{ex1}}
\label{fig2}
\end{figure}

Note that $T_{\al,\beta}$ has an indifferent fixed point at $0$ for each $\al\in\mb{A}$ and $\beta\in\mb{B}$.
$T_{\al,\beta}$ also has derivative $0$ at $\frac{1}{2}$, around an inverse image of the indifferent fixed point, for any $\al>0$ and $\beta>1$ (see also Figure \ref{fig2}).
According to the asymptotic equation \eqref{eq1} we have
\[
y_{n+1}^{\bdsy{\al},\beta}-\frac{1}{2}\approx n^{-\frac{1}{\al\beta}}.
\]
Then applying Theorem \ref{Cor2} to this model, since $\eta(\bdsy{\al},\beta)=0$ for each $\bdsy{\al}\in\mb{A}^{\N}$ and $\beta\in\mb{B}$, we have the following.

\begin{proposition}\label{prop1}
The random piecewise convex map derived from \eqref{ex1} admits a $\lam$-equivalent, conservative and ergodic $\sigma$-finite invariant measure $\mu$ such that for any $\al'\in\mb{A}$ with $\nu_{\mb{A}}\{\al\in\mb{A}:\al'\le\al\}>0$
\[
\int_{\mb{B}}n^{-\frac{1}{\al_0\beta}}d\nu_{\mb{B}}(\beta)\lessapprox\mu\left(X_n^{\bdsy{\al_0}}\right)
\quad\text{and}\quad
\mu\left(X_n^{\bdsy{\al'}}\right)\lessapprox\int_{\mb{B}}n^{-\frac{1}{\al'\beta}}d\nu_{\mb{B}}(\beta)
\]
for $n$ large enough.
\end{proposition}

\begin{remark}
(I)
From Proposition \ref{prop1}, if $\nu_{\mb{B}}\left\{\beta\in\mb{B}:\al_0\beta\ge1\right\}>0$ then $\mu([0,1])=\infty$ and if $\al'<(\max_{\mb{B}}\beta)^{-1}=1$ for some $\al'\in\mb{A}$ then $\mu([0,1])<\infty$.
We remark that the invariant measure of \eqref{ex1} tends to become infinite rather than that of \eqref{ex0}.

(II)
\cite[Theorem 1.1]{C} showed that an upper bound for the invariant density
\[
\frac{d\mu}{d\lam}(x)\lessapprox x^{-(1+\al-\frac{1}{\beta})}
\]
holds for the deterministic map (\ref{ex1}) with a fixed parameter such that $1<\beta<\frac{1}{\al}$ (and hence only finite invariant measures are dealt with in \cite{C}).
This also implies that $\mu(X_n^{\bdsy{\al}})\lessapprox n^{-\frac{1}{\al\beta}}$.
Thus Proposition \ref{prop1} is a random generalization of \cite{C} (note that our result can admit parameter $\beta\ge\frac{1}{\al}$) as well as showing lower bound of $\mu$ for large $n\ge1$.
\end{remark}

\subsection{Random weakly expanding maps with a flat point}\label{subsec45}

Let $\mb{A}\subset(0,+\infty)$ and $\mb{B}\subset[1,+\infty)$\footnote{If $\beta<1$, then the convex property of $S_{\beta}$ may be violated.} compact sets and $\nu_{\mb{A}}$ and $\nu_{\mb{B}}$ be probability measures on $\mb{A}$ and $\mb{B}$, respectively.
For $\al\in\mb{A}$ and $\beta\in\mb{B}$, define
\begin{linenomath}
\begin{align}\label{ex2}
T_{\al,\beta}x=
\begin{cases}
x\left(1+2^{\al}x^{\al}\right) &x\in\left[0,\frac{1}{2}\right],\\
\exp\left( 2^{\beta}-\left(x-\frac{1}{2}\right)^{-\beta}\right) &x\in\left(\frac{1}{2},1\right].
\end{cases}
\end{align}
\end{linenomath}
Then we can see that $\frac{1}{2}$, the inverse image of $0$, is a flat point in the sense that $T_{\al,\beta}^{(n)}(\frac{1}{2})=0$ for any $\al>0$ and $\beta\ge1$.
Using the same notation, we have
\[
y_{n+1}^{\bdsy{\al},\beta}-\frac{1}{2}\approx \left(\log n\right)^{-\frac{1}{\beta}}
\]
for large $n\ge1$.
We can again apply Theorem \ref{Cor2} to this model.

\begin{proposition}\label{prop2}
The random piecewise convex map derived from \eqref{ex2} admits a $\lam$-equivalent, conservative and ergodic $\sigma$-finite invariant measure $\mu$ such that for any $\bdsy{\al}=(\al,\al,\dots)\in\mb{A}$
\[
\mu\left(X_n^{\bdsy{\al}}\right)\approx \int_{\mb{B}}\left(\log n\right)^{-\frac{1}{\beta}}d\nu_{\mb{B}}(\beta)
\]
for $n$ large enough.
Consequently, we always have $\mu([0,1])=\infty$ for any $\mb{A}$, $\mb{B}$, $\nu_{\mb{A}}$ and $\nu_{\mb{B}}$.
\end{proposition}

\begin{remark}
We remark that even a modification of $\al\to0$ leaves no space for $\mu$ to be finite.
That is, if $\mb{A}=\{0\}$ (so we drop a symbol $\al$ henceforth in this remark) and $\nu_{\mb{A}}$ is a point mass on $\mb{A}$ then $T_{\beta}x=2x$ for $x\in[0,\frac{1}{2}]$ and $y_{n+1}^{\beta}-\frac{1}{2}\approx n^{-\frac{1}{\beta}}$.
This means that $\mu(X_n)\approx\int_{\mb{B}}n^{-\frac{1}{\beta}}d\nu_{\mb{B}(\beta)}$ and we still always have $\mu([0,1])=\infty$ because $\min_\mb{B}\beta\ge1$.
\end{remark}

\subsection{Random weakly expanding maps with a wide entrance}\label{subsec46}
Our example is defined as follows, which is similar to the examples (\ref{ex-1}) and (\ref{ex1}).
Let $\mb{A}=[0,\al_1]$ for some $0<\al_1<\infty$ and $\mb{B}=[1,+\infty)$ and $\nu_{\mb{A}}$ and $\nu_{\mb{B}}$ be probability measures on $\mb{A}$ and $\mb{B}$, respectively.
For $\al\in\mb{A}$ and $\beta\in\mb{B}$, define
\begin{linenomath}
\begin{align}\label{ex3}
T_{\beta}x=
\begin{cases}
x\left(1+2^{\al}x^{\al}\right) &x\in\left[0,\frac{1}{2}\right],\\
\left(x-\frac{1}{2}\right)^{\beta} &x\in\left(\frac{1}{2},1\right].
\end{cases}
\end{align}
\end{linenomath}
From the definition $T_{\beta}(1)=2^{-\beta}$ and hence the image of the right half part $(\frac{1}{2},1]$ will vanish as $\beta$ tends to infinity.
Again, let $\mb{B}_k(\bdsy{\al})=\{\beta\in\mb{B}:\eta(\bdsy{\al},\beta)=k\}$ and $\mb{B}=\bigcup_{k=1}^{\infty}\mb{B}_k(\bdsy{\al})$ (disjoint) for each $\bdsy{\al}=(\al,\al,\dots)\in\mb{A}^{\N}$.

We consider two cases of $\al=0$ and $\al>0$, and we first observe when $\al=0$ which gives a lower bound for the invariant measure for the random piecewise convex map given by \eqref{ex3}.
In this case, it is straightforward to see that for each $\beta\in\mb{B}_k(\bdsy{0})$
\[
x_n^{\bdsy{0}}=2^{-n}
\quad\text{and}\quad
y_{n+1}^{\bdsy{0},\beta}-\tfrac{1}{2}=2^{-\frac{n+k}{\beta}}
\]
with notation $\bdsy{0}=(0,0,\dots)\in\mb{A}^{\N}$.
We also have $\mb{B}_k(\bdsy{0})=[k,k+1)$ for $k\ge1$.
Thus the invariant measure $\mu_{\bdsy{0}}$ for a random piecewise convex map $\{T_{0,\beta};\nu_{\mb{B}}:\beta\in\mb{B}\}$ satisfies that
\begin{linenomath}
\begin{align*}
\mu_{\bdsy{0}}\left(X_n^{\bdsy{0}}\right)
&\approx\sum_{k=1}^{n-1}\int_{\mb{B}_k(\bdsy{0})}\left(y_{n-k+1}^{\bdsy{0},\beta}-\tfrac{1}{2}\right)d\nu_{\mb{B}}(\beta) + \sum_{k=n}^{\infty}\nu_{\mb{B}}\left(\mb{B}_k(\bdsy{0})\right)\\
&=\sum_{k=1}^{n-1}\int_{[k,k+1)}2^{-\frac{n}{\beta}}d\nu_{\mb{B}}(\beta) + \nu_{\mb{B}}\left([n,\infty)\right)\\
&=\int_{[1,n)}2^{-\frac{n}{\beta}}d\nu_{\mb{B}}(\beta) + \nu_{\mb{B}}\left([n,\infty)\right)
\end{align*}
\end{linenomath}
for large $n$ by Theorem \ref{Cor1}.

We secondly consider the case when $\al>0$ which is in need for an upper bound for the invariant measure.
Then for each $\bdsy{\al}=(\al,\al,\dots)\in\mb{A}^{\N}$ and $\beta\in\mb{B}_k(\bdsy{\al})$ we have
\[
y_{n+1}^{\bdsy{\al},\beta}-\frac{1}{2}\approx (n+k)^{-\frac{1}{\al\beta}}
\]
as $n\to\infty$.
Then the invariant measure $\mu_{\bdsy{\al}}$ for the random piecewise convex map $\{T_{\al,\beta};\nu_{\mb{B}}:\beta\in\mb{B}\}$ satisfies that
\begin{linenomath}
\begin{align*}
\mu_{\bdsy{\al}}\left(X_n^{\bdsy{\al}}\right)
&\approx\sum_{k=1}^{n-1}\int_{\mb{B}_k(\bdsy{\al})} n^{-\frac{1}{\al\beta}} d\nu_{\mb{B}}(\beta)
+\sum_{k=n}^{\infty}\nu_{\mb{B}}\left(\mb{B}_k(\bdsy{\al})\right)\\
&\approx \int_{\left[1,\sup_{\mb{B}_{n-1}(\bdsy{\al})}\beta\right)} n^{-\frac{1}{\al\beta}} d\nu_{\mb{B}}(\beta)
+\sum_{k=n}^{\infty}\nu_{\mb{B}}\left(\mb{B}_k(\bdsy{\al})\right)
\end{align*}
\end{linenomath}
From these observations, we have the following.

\begin{proposition}\label{prop4666}
The random piecewise convex map derived from \eqref{ex3} admits a $\lam$-equivalent, conservative and ergodic $\sigma$-finite invariant measure $\mu$ such that for any $\al'\in\mb{A}$ with $\nu_{\mb{A}}\{\al\in\mb{A}:\al'\le\al\}>0$
\begin{linenomath}
\begin{align*}
&\int_{[1,n)}2^{-\frac{n}{\beta}}d\nu_{\mb{B}}(\beta) + \nu_{\mb{B}}\left([n,\infty)\right)
\lessapprox\mu\left(X_n^{\bdsy{0}}\right)
\quad\text{and}\\
&\qquad\qquad\qquad\mu\left(X_n^{\bdsy{\al'}}\right)\lessapprox
\int_{\left[1,\inf_{\mb{B}_{n}(\bdsy{\al'})}\beta\right)} n^{-\frac{1}{\al'\beta}} d\nu_{\mb{B}}(\beta)
+\nu_{\mb{B}}\left(\bigg[\inf_{\mb{B}_n(\bdsy{\al'})}\beta,\infty\bigg)\right)
\end{align*}
\end{linenomath}
for $n$ large enough, where $\bdsy{0}=(0,0,\dots)\in\mb{A}^{\N}$ and $\bdsy{\al'}=(\al',\al',\dots)\in\mb{A}^{\N}$.
\end{proposition}

\begin{remark}
As an example of this proposition, if $d\nu_{\mb{B}}(\beta)=(\ell-1)\beta^{-\ell}d\beta$ for some $\ell>1$ on $\mb{B}=[1,+\infty)$ then the invariant measure $\mu$ is infinite when $\ell\le2$, independent of the choice of $\nu_{\mb{A}}$.
The calculation is similar as in Remark \ref{rem42}:
To see this, we need to show the lower bound in Proposition \ref{prop4666} is proportional to or greater than $n^{-1}$.
\begin{linenomath}
\begin{align*}
\int_{[1,n)}2^{-\frac{n}{\beta}}d\nu_{\mb{B}}(\beta) + \nu_{\mb{B}}\left([n,\infty)\right)
&\ge 2^{-n}\int_1^n (\ell-1)\beta^{-\ell}d\beta +\int_n^{\infty} (\ell-1)\beta^{-\ell}d\beta\\
&= 2^{-n}\left(1-n^{-(\ell-1)}\right)+n^{-(\ell-1)}\\
&> n^{-1}
\end{align*}
\end{linenomath}
and our conclusion is valid.
\end{remark}

\subsection{Counterexample with infinite derivative}\label{counter}
We finally illustrate an example which does  not satisfy \eqref{3} in Theorem \ref{Thm1}.
The example below still admits an equivalent $\sigma$-finite invariant measure but the density function of the invariant measure is no longer bounded away from zero.

Let $\mb{B}$ be a parameter space and $\nu_{\mb{B}}$ be a probability measure on $\mb{B}$.
For $\beta\in\mb{B}$, define
\begin{linenomath}
\begin{align}\label{ce1}
T_{\beta}x=
\begin{cases}
1-\sqrt{1-2x} &\left[0,\frac{1}{2}\right],\\
S_{\beta} &\left(\frac{1}{2},1\right]
\end{cases}
\end{align}
\end{linenomath}
where $S_{\beta}$'s satisfy \eqref{0}--\eqref{2}.
Note that \eqref{3} does not hold, while \eqref{4} holds.
We then assume that there is some $\kappa\in(0,1)$ such that for $\nu_{\mb{B}}$-almost every $\beta\in\mb{B}$, $S_{\beta}(1)\le 1-\kappa$ holds.

Since \eqref{ce1} satisfies \eqref{4}, there is an equivalent $\sigma$-finite invariant measure $\mu$ for this random piecewise convex map.
Then we have for any $0<\vep<\kappa$
\[
\mu\left([1-\vep,1]\right)
=\int_{\mb{B}}\mu\left(T_{\beta}^{-1}[1-\vep,1]\right)d\nu_{\mb{B}}(\beta)
=\mu\left(\left[\tfrac{1-\vep^2}{2},\tfrac{1}{2}\right]\right).
\]
Since $\frac{d\mu}{d\lam}\le C_1$ on $[\frac{1-\kappa^2}{2},1]$ for some $C_1>0$,
\[
\mu\left(\left[\tfrac{1-\vep^2}{2},\tfrac{1}{2}\right]\right)\le \frac{C_1\vep^2}{2}.
\]
If $\frac{d\mu}{d\lam}$ were bounded away from zero on $Y$, then we also have
\[
\mu\left([1-\vep,1]\right)\ge C_2\vep
\]
for some $C_2>0$.
However this implies $\vep\ge\frac{2C_2}{C_1}$, which is contradiction since $\vep>0$ can be taken arbitrary small.
Therefore, we conclude that $\frac{d\mu}{d\lam}(x)\to0$ as $x\to 1$.

\subsection*{Acknowledgements}
This work was supported by the Research Institute for Mathematical Sciences, an International Joint Usage/Research Center located in Kyoto University.
Hisayoshi Toyokawa was supported by JSPS KAKENHI Grant Number 21K20330.


\end{document}